\documentclass[10pt]{amsart}

\usepackage{amsfonts,amssymb,enumerate,color,bm,hhline,makecell,array,mathtools}
\usepackage{array}
\usepackage{mathrsfs}
\usepackage{comment}
\usepackage[all,ps,cmtip]{xy} 
\usepackage[normalem]{ulem}
\usepackage[colorlinks=true,citecolor=blue, urlcolor=blue, linkcolor=blue]{hyperref}
\usepackage{tikz-cd}
\usepackage{amscd}
\usepackage[shortlabels]{enumitem}
\usepackage{tensor}
\usepackage{tikz}
\usetikzlibrary{matrix,arrows}
\usepackage{extarrows}
\usepackage{stmaryrd} 
\usepackage[new]{old-arrows} 

\usepackage{xcolor}

\newcommand{\CC}{{\mathbb C}}

\newcommand{\NN}{{\mathbb N}}
\newcommand{\PP}{{\mathbb P}}
\newcommand{\QQ}{{\mathbb Q}}
\newcommand{\RR}{{\mathbb R}}
\newcommand{\ZZ}{{\mathbb Z}}

\def\C{\mathbb{C}}

\def\P{\mathbb{P}}
\def\Q{\mathbb{Q}}

\def\Z{\mathbb{Z}}

\let\mathcal\mathscr

\def\cC{\mathcal{C}}
\def\cD{\mathcal{D}}
\def\cE{\mathcal{E}}
\def\cF{\mathcal{F}}
\def\cG{\mathcal{G}}

\def\cI{\mathcal{I}}

\def\cL{\mathcal{L}}
\def\cM{\mathcal{M}}
\def\cN{\mathcal{N}}
\def\cO{\mathcal{O}}

\def\cQ{\mathcal{Q}}

\def\cX{\mathcal{X}}
\def\cY{\mathcal{Y}}

\def\vv{\ensuremath{\mathbf v}}

\def\phi{{\varphi}}

\DeclareMathOperator{\Bl}{Bl}

\DeclareMathOperator{\ch}{ch}
\DeclareMathOperator{\CH}{CH}
\DeclareMathOperator{\cl}{cl}

\DeclareMathOperator{\coh}{coh}

\DeclareMathOperator{\coker}{Coker}

\DeclareMathOperator{\Def}{Def}
\DeclareMathOperator{\Diag}{Diag}
\DeclareMathOperator{\Db}{D\textsuperscript{\rm b}}

\def\div{\mathop{\rm div}\nolimits}
\DeclareMathOperator{\divisore}{div}

\DeclareMathOperator{\End}{End}

\DeclareMathOperator{\Ext}{Ext}
\DeclareMathOperator{\Fix}{Fix}
\DeclareMathOperator{\GL}{GL}
\DeclareMathOperator{\Gr}{Gr}

\DeclareMathOperator{\Hom}{Hom}

\def\Im{\mathop{\rm Im}\nolimits}

\DeclareMathOperator{\NS}{NS}

\DeclareMathOperator{\Pic}{Pic}

\DeclareMathOperator{\Proj}{Proj}

\DeclareMathOperator{\Spec}{Spec}

\DeclareMathOperator{\Sing}{Sing}

\DeclareMathOperator{\Sym}{Sym}

\DeclareMathOperator{\td}{td}

\newcommand{\es}{{\emptyset}}

\newcommand{\la}{\langle}
\newcommand{\ov}{\overline}
\newcommand{\ra}{\rangle}
\newcommand{\wh}{\widehat}
\newcommand{\wt}{\widetilde}

\def\lra{\longrightarrow}
\def\llra{\hbox to 10mm{\rightarrowfill}}
\def\lllra{\hbox to 15mm{\rightarrowfill}}

\def\llla{\hbox to 10mm{\leftarrowfill}}
\def\lllla{\hbox to 15mm{\leftarrowfill}}

\def\dra{\dashrightarrow}

\def\hra{\hookrightarrow}

\newtheorem{lemm}{Lemma}[section]
\newtheorem{theo}[lemm]{Theorem}
\newtheorem*{MainThm*}{Main Theorem}
\newtheorem{coro}[lemm]{Corollary}
\newtheorem{prop}[lemm]{Proposition}

\theoremstyle{definition}
\newtheorem{defi}[lemm]{Definition}
\newtheorem{rema}[lemm]{Remark}
\newtheorem{conj}[lemm]{Conjecture}

\newtheorem{exam}[lemm]{Example}

\theoremstyle{remark}
\newtheorem*{remark*}{Remark}
\newtheorem*{note*}{Note}
\newtheorem*{lemmared*}{Lemma}

\makeatletter
\@namedef{subjclassname@2020}{%
  \textup{2020} Mathematics Subject Classification}
\makeatother

\makeatletter
\def\@tocline#1#2#3#4#5#6#7{\relax
  \ifnum #1>\c@tocdepth 
  \else
    \par \addpenalty\@secpenalty\addvspace{#2}%
    \begingroup \hyphenpenalty\@M
    \@ifempty{#4}{%
      \@tempdima\csname r@tocindent\number#1\endcsname\relax
    }{%
      \@tempdima#4\relax
    }%
    \parindent\z@ \leftskip#3\relax \advance\leftskip\@tempdima\relax
    \rightskip\@pnumwidth plus4em \parfillskip-\@pnumwidth
    #5\leavevmode\hskip-\@tempdima
      \ifcase #1
       \or\or \hskip 1em \or \hskip 2em \else \hskip 3em \fi%
      #6\nobreak\relax
    \dotfill\hbox to\@pnumwidth{\@tocpagenum{#7}}\par
    \nobreak
    \endgroup
  \fi}
\makeatother

\allowdisplaybreaks

\usepackage{enumitem}
\setlist[itemize]{noitemsep,nolistsep}
\setlist[enumerate]{noitemsep,nolistsep}
\usepackage{colortbl}
\usepackage{enumerate}

\def\citestacks#1{\cite[\href{https://stacks.math.columbia.edu/tag/#1}{Tag #1}]{stacks-project}}

\title{The geometry of antisymplectic involutions, II}

\begin{document}

\author{Laure Flapan}
\address{\parbox{0.9\textwidth}{Michigan State University, Department of Mathematics\\[1pt]
619 Red Cedar Road, East Lansing, MI 48824, USA
\vspace{1mm}}}
\email{{flapanla@msu.edu}}

\author{Emanuele Macr\`i}
\address{\parbox{0.9\textwidth}{Universit\'e Paris-Saclay, CNRS, Laboratoire de Math\'ematiques d'Orsay\\[1pt]
Rue Michel Magat, B\^at. 307, 91405 Orsay, France
\vspace{1mm}}}
\email{{emanuele.macri@universite-paris-saclay.fr}}

\author{Kieran G.~O'Grady}
\address{\parbox{0.9\textwidth}{Sapienza Universit\`a di Roma, Dipartimento di Matematica\\[1pt]
P.le A. Moro 5, 00185 Roma, Italia
\vspace{1mm}}}
\email{{ogrady@mat.uniroma1.it}}

\author{Giulia Sacc\`a}
\address{\parbox{0.9\textwidth}{Columbia University, Department of Mathematics\\[1pt]
2990 Broadway, New York, NY 10027, USA
\vspace{1mm}}}
\email{{gs3032@columbia.edu}}

\thanks{L.F.~is partially supported by NSF grant DMS-2200800. E.M.~is partially supported by the ERC Synergy Grant ERC-2020-SyG-854361-HyperK. K.O'G.~is partially supported by the PRIN 2017YRA3LK \lq\lq Moduli and Lie Theory\rq\rq.  G.S.~is partially supported by the NSF CAREER grant DMS-2144483 and by the NSF FRG grant DMS-2052750.}

\begin{abstract}
We continue our study  of fixed loci of antisymplectic involutions on projective hyper-K\"ahler manifolds of $\mathrm{K3}^{[n]}$-type induced by an ample class of square 2 in the Beauville-Bogomolov-Fujiki lattice.
We prove that if the divisibility of the ample class is~2, then one connected component of the fixed locus is a Fano manifold of index~3, thus generalizing  to higher dimensions the case of the LLSvS 8-fold associated to a cubic fourfold.
We also show that, in the case of the LLSvS 8-fold associated to a cubic fourfold, the second component of the fixed locus is of general type, thus answering a question by Manfred Lehn.
\end{abstract}

\maketitle

\tableofcontents
\setcounter{tocdepth}{1}


\section{Introduction}\label{sec:intro}

Many of the known explicit constructions of projective hyper-K\"ahler ({\rm HK}) manifolds of \emph{$\mathrm{K3}^{[n]}$-type} (i.e., a  deformation of the Hilbert scheme of $n$ points on a K3 surface) arise from a Fano variety.
Classically, Beauville--Donagi constructed a maximal family of polarized HK manifolds of $\mathrm{K3}^{[2]}$-type  as Fano varieties of lines on  cubic fourfolds \cite{BD:Fano}. 
More recently, Lehn--Lehn--Sorger--van Straten constructed (see Example \ref{ex:LLSvS} below) a maximal family of  polarized HK manifolds of $\mathrm{K3}^{[4]}$-type as equivalence classes of twisted cubics on cubic fourfolds~\cite{LLSvS:cubics}. 
Other explicit constructions of families of polarized HK manifolds of $\mathrm{K3}^{[n]}$-type from Fano varieties arise for instance in \cite{DV: varieties,IM:Fano, DK: GM varieties,IM: magic square,FM:Fano}. 
In the above examples, there is a tight relationship between the Hodge theory of the underlying Fano variety, or at a finer level its derived category, and the geometry of the resulting HK manifold, see also for instance \cite{BLMNPS:family,PPZ: GM,LPZ: twisted}.

The present paper aims to strengthen and formalize these observed connections between Fano varieties and HK manifolds of $\mathrm{K3}^{[n]}$-type by proposing the following partial inverse of the above constructions.
We begin with a projective hyper-K\"ahler manifold $X$ of $\mathrm{K3}^{[n]}$-type equipped with an ample class $\lambda$ of square $2$. There is a unique involution $\tau_{\lambda}$ of $X$ whose action on $H^2(X)$ is equal to the reflection in $\lambda$, see equation~\eqref{riflessione} below.  Since $\tau_{\lambda}$ is antisymplectic, its fixed locus $\Fix(\tau_{\lambda})$ is Lagrangian. 
We showed in ~\cite{involutions1} that the number of connected components of $\Fix(\tau_{\lambda})$ is equal to the divisibility $\divisore(\lambda)$ of $\lambda$  in the Beauville--Bogomolov--Fujiki lattice (in Section~\ref{subsec:AntisymplecticIntro} we recall the notion of divisibility). Note that since $\lambda$ has square 2, its divisibility $\divisore(\lambda)$ must be either $1$ or $2$. 
The first main result of the present paper (see Theorem \ref{thm:main1} below) shows that in the case that $\divisore(\lambda)=2$, one of the two connected components of $\Fix(\tau_{\lambda})$ is a Fano variety (of index 3). 

\begin{exam}\label{ex:LLSvS}
The simplest example with  $\lambda$ of divisibility $2$ is given by the LLSvS $8$-fold $X$ associated to a smooth cubic fourfold $W\subset\PP^5$ containing no planes, parametrizing equivalence classes of twisted cubic curves contained in $W$, see~\cite{LLSvS:cubics}. 
The polarization $\lambda$ is the pull-back of the Pl\"ucker polarization on $\Gr(3,\PP^5)$ via the rational map $X\dra \Gr(3,\PP^5)$ which maps a twisted cubic to its linear span (equivalent cubics span the same $3$ dimensional linear subspace).
For a description of the involution $\tau_\lambda$ in this case see~\cite{Lehn:Oberwolfach,CCL:Chow}. It is known that all  HK manifolds $X$ of K3$^{[4]}$-type with a polarization $\lambda$ of square $2$ and divisibility $2$ arise as LLSvS $8$-folds in this way, see~\cite[Proposition B.12]{Debarre:Survey}. On these $8$-folds,  one connected component of $\Fix(\tau_{\lambda})$ is isomorphic to the cubic fourfold $W$, a Fano variety of index $3$. 
\end{exam}

Our second main result, Theorem \ref{thm:main2} below, establishes that the other connected component of $\Fix(\tau_{\lambda})$ in Example~\ref{ex:LLSvS} has ample canonical bundle, thereby answering a question by Manfred Lehn. In fact we  conjecture that the other connected component of $\Fix(\tau_{\lambda})$ when $\divisore(\lambda)=2$, as well as the only 
connected component of $\Fix(\tau_{\lambda})$ when $\divisore(\lambda)=1$, has ample canonical bundle. The  conjecture is motivated by Theorem \ref{thm:main2} and by the following examples.

\begin{exam}\label{ex:divuno}
The simplest example with  $\lambda$ of divisibility $1$ is provided by the double cover $\pi\colon X\to\PP^2$ branched over a smooth sextic curve, with $\lambda=\pi^{*}c_1(\cO_{\PP^2}(1))$ (thus $X$ is a {\rm K3} surface). The involution $\tau_\lambda$ interchanges the sheets of the double cover $\pi$, and the fixed locus of $\tau_\lambda$ is isomorphic to the branch curve, i.e., a smooth plane sextic curve. The next simplest example 
with  $\lambda$ of divisibility $1$ is provided by the double cover $f\colon X\to Y$ of a general EPW sextic $Y\subset \PP^5$, with $\lambda=f^{*}c_1(\cO_{Y}(1))$. The involution $\tau_\lambda$ is the covering involution, and the fixed locus is the singular locus of $Y$, a surface 
of general type by the results in~\cite{Kieran:numhilb2}. 
\end{exam}

We also expect a tight connection between the Hodge structure of a polarized hyper-K\"ahler manifold $(X,\lambda)$ 
of $\mathrm{K3}^{[n]}$-type
with $\lambda$ of square $2$ and divisibility $2$, and the Hodge structure of the Fano connected component of 
$\Fix(\tau_{\lambda})$. At a finer level, their bounded derived categories  of coherent sheaves should be strictly related. These developments will be the subject of future work.

Let us explain at least one reason for being interested in the component of $\Fix(\tau_{\lambda})$ (conjecturally) of general-type.  For $(X,\lambda)$ a polarized {\rm HK} manifold of {\rm K3}$^{[n]}$-type with $q_X(\lambda)=2$ and $\divisore(\lambda)=1$ and $n\in\{1,2\}$ it is known that a multiple of the cycle $\Fix(\tau_{\lambda})$ moves in a covering family of Lagrangian subvarieties of $X$ (if $n=1$ this is true because $\Fix(\tau_{\lambda})$ is a very ample divisor; if $n=2$ this is proved in~\cite{Ferretti:thesis,Kieran:LagrangianCovering}). Thus, 
an interesting question for arbitrary $n$, both in the case $\divisore(\lambda)=1$ and $\divisore(\lambda)=2$, is whether a multiple of this conjecturally general-type component of $\Fix(\tau_{\lambda})$ could provide a
 covering family of Lagrangian subvarieties in $X$.
This is open already in the  case $n=4$ and $\divisore(\lambda)=2$ (the LLSvS variety).

\subsection{Antisymplectic involutions}\label{subsec:AntisymplecticIntro}

Here and in what follows $X$ is a {\rm HK}  manifold of $\mathrm{K3}^{[n]}$-type, in particular it has dimension $2n$. 
We let $q_X$ be the Beauville--Bogomolov--Fujiki quadratic form on $H^2(X)$, and by abuse of notation we denote by the same letter the associated bilinear symmetric form. 
If $\alpha\in H^2(X,\ZZ)$, then the \emph{divisibility} of $\alpha$ is the non negative generator of the ideal $\{q_X(\alpha,w)\,:\,w\in H^2(X,\ZZ)\}\subset\ZZ$; it is denoted by $\divisore(\alpha)$. 

We assume that $X$ has a polarization $\lambda$ such that $q_X(\lambda)=2$. Note that $\divisore(\lambda)\in\{1,2\}$.  For each $n$ there exist $(X,\lambda)$ as above with 
$\divisore(\lambda)=1$. On the other hand there exist $(X,\lambda)$ as above with 
$\divisore(\lambda)=2$ if and only if 
 $4\,|\,n$.
To a polarization as above  one can associate a (unique) antisymplectic involution
\[
\tau_{\lambda}\colon X\xlongrightarrow{\cong}X
\]
which acts on $H^2(X)$ as reflection in $\lambda$:
\begin{equation}\label{riflessione}
\begin{matrix}
H^2(X) & \overset{\tau^{*}_{\lambda}}{\lra} & H^2(X) \\
x & \longmapsto & -x+q_X(\lambda,x)\lambda. 
\end{matrix}
\end{equation}
Furthermore there exists (a unique) lift of the involution $\tau_{\lambda}$ to an involution of the line bundle $L$ on $X$ such that $\lambda=c_1(L)$: this 
defines a $\mu_2$-action on $L$ which acts trivially on $H^0(X,L)$ (this is proved in Proposition~\ref{prop:ChoiceLinearizationFamily} for the case $\divisore(\lambda)=2$ and in Proposition~\ref{prop:sezinv} for the case $\divisore(\lambda)=1$).

Now let $\Fix(\tau_{\lambda})$ be the fixed locus  of $\tau_{\lambda}$, a Lagrangian submanifold of $X$. 
In~\cite{involutions1} we proved that  the number of connected components of $\Fix(\tau_{\lambda})$ is equal to $\divisore(\lambda)$. 
By considering the $\mu_2$-action on $L$ introduced above, we can be more precise. 
Let  $\C_{\pm}$ be the trivial, respectively determinantal, irreducible $\mu_2$-representation, and set
\begin{equation*}\label{eq:SigmaPM}
\begin{split}
&\Fix(\tau_{\lambda})_+ :=\{ x\in \Fix(\tau_\lambda) \,:\, L\vert_x \cong \C_+\},\quad \text{the \emph{positive} fixed component}\\
&\Fix(\tau_{\lambda})_- :=\{ x\in \Fix(\tau_\lambda) \,:\, L\vert_x \cong \C_-\},\quad \text{the \emph{negative} fixed component}.
\end{split}
\end{equation*}
Then $\Fix(\tau_{\lambda})=\Fix(\tau_{\lambda})_+$ if $\divisore(\lambda)=1$, and $\Fix(\tau_{\lambda})_+,\Fix(\tau_{\lambda})_-$ are the two irreducible components of the fixed locus $\Fix(\tau_{\lambda})$ if $\divisore(\lambda)=2$.

\subsection{The negative component}\label{subsec:Fano}

Our first main result states that if $\divisore(\lambda)=2$ then the negative component $\Fix(\tau_{\lambda})_-$ is a Fano manifold of index~3.

\begin{theo}\label{thm:main1}
Let $(X,\lambda)$ be a polarized {\rm HK} manifold of {\rm K3}$^{[n]}$-type, with $q_X(\lambda)=2$ and $\divisore(\lambda)=2$ (hence $4\,|\,n$). 
Let $L$ be the ample line bundle such that $c_1(L)=\lambda$. 
Then $\Fix(\tau_{\lambda})_-$ is a Fano manifold of dimension~$n$ and index~$3$. More precisely, we have
\begin{equation}\label{canmenotre}
\omega_{\Fix(\tau_{\lambda})_-}\cong (L^{\vee})^{\otimes 3} \vert_{\Fix(\tau_{\lambda})_-}.
\end{equation}
\end{theo}

\subsection{The positive component}\label{subsec:LehnQuestion}

We conjecture that the  component $\Fix(\tau_{\lambda})_+$ is of general type, both when $\divisore(\lambda)=1$ and $\divisore(\lambda)=2$.
More precisely we expect the canonical line bundle of $\Fix(\tau_{\lambda})_+$ to be a rational positive multiple of the restriction of $\lambda$.

Our second main result  confirms the above expectation in the case $n=4$ and $\divisore(\lambda)=2$.

\begin{theo}\label{thm:main2}
Let $(X,\lambda)$ be a polarized {\rm HK} manifold of dimension $8$ of {\rm K3}$^{[4]}$-type, with $q_X(\lambda)=2$ and $\divisore(\lambda)=2$. Let $L$ be the ample line bundle such that $c_1(L)=\lambda$.
Then
\[
\omega_{\Fix(\tau_{\lambda})_+}\cong L^{\otimes 3}\vert_{\Fix(\tau_{\lambda})_+}.
\]
In particular $\Fix(\tau_{\lambda})_+$ is of general-type.
\end{theo}

Let $(X,\lambda)$ be as in Theorem~\ref{thm:main2}. 
Then there exists a smooth cubic fourfold $W\subset\PP^5$  containing no planes such that $X$ is isomorphic to the LLSvS $8$-fold  associated to $W$, and the isomorphism may be chosen so that  $\lambda$ is identified with the polarization described in Example \ref{ex:LLSvS}.  
The generic element in the positive component $\Fix(\tau_{\lambda})_+$ is the equivalence class of twisted cubics whose associated cubic surface has four $A_1$ singularities (see, e.g.,~\cite{Lehn:Oberwolfach}).
Manfred Lehn asked about the global geometry of this component: Theorem~\ref{thm:main2} gives a first answer to his question.
In Section~\ref{sec:poscomp}, we discuss the nature of the  fixed locus (=positive fixed component) when $q_X(\lambda)=2$ and $\div(\lambda)=1$. 
As discussed in Example~\ref{ex:divuno}, it is of general type  
if $n=1$ or $n=2$. 
We show for all $n$ that the fixed locus is of general type provided the linear system $|\lambda|$ behaves well in a neighborhood of the fixed locus.

\subsection{Ideas from the proofs}\label{subsec:IntroIdeaProof}

As in our previous work~\cite{involutions1}, the basic idea for both Theorems \ref{thm:main1} and \ref{thm:main2} is to look at singular {\rm HK} varieties with an involution which are specializations of {\rm HK} manifolds of {\rm K3}$^{[n]}$-type with an involution induced by a polarization $\lambda$ with $q_X(\lambda)=2$ and $\divisore(\lambda)=2$. 
In~\cite{involutions1} we studied  specializations given by the images of divisorial contractions of  HK manifolds with a (rational) Lagrangian fibration. 
In the present paper we study  specializations with worse singularities, given by  the images of divisorial contractions of moduli spaces of stable rank $2$ sheaves on a polarized K3 surface $S$ of degree~2. 
More specifically, we consider (singular) Donaldson--Uhlenbeck--Yau compactifications of moduli spaces of slope-stable rank-2 sheaves on $S$. Let $f\colon S\to \PP^2$ be \lq\lq the\rq\rq\ double cover associated to the polarization of degree $2$ on $S$. 
The involution of the Donaldson--Uhlenbeck--Yau moduli  space $\ov{M}_{S,n}$ to which $\tau_\lambda$ specializes is induced by the covering involution of the double cover $f$. 
Moreover, one of the (several) components of the fixed locus of this involution of  $\ov{M}_{S,n}$  is identified, set theoretically,  with a Donaldson--Uhlenbeck--Yau moduli space $\ov{M}_{\PP^2,n}$ (see Section~\ref{subsec:DUYP2} for notation) of slope-stable rank-2 sheaves on the projective plane (see Theorem~\ref{thm:embedding}). 
More precisely, the pull-back $f^{*}$ defines a regular bijective map from the normal variety $\ov{M}_{\PP^2,n}$ to a component of the fixed locus.

The reason for using a different specialization in the present paper compared to that used in~\cite{involutions1} is that, while the fixed locus of the involution on $\ov{M}_{S,n}$ becomes more complicated, the only relevant component for the proof of Theorem~\ref{thm:main1} is the one corresponding to  $\ov{M}_{\PP^2,n}$ which, as is well-known, is a Fano variety of index $3$.
The hard and technical part of the proof is to deal with the singularities of the Donaldson--Uhlenbeck--Yau moduli space, in particular proving that the component of the fixed locus mentioned above is normal, and hence isomorphic to $\ov{M}_{\PP^2,n}$. 
We first need a local description of the moduli space: this is done by proving that, in our case, the Donaldson--Uhlenbeck--Yau moduli space is \emph{isomorphic} to a moduli space of Bridgeland stable objects (\cite{Bridgeland:Stab,Tajakka:Uhlenbeck}).
We finish the proof of the normality result by using this description, together with some fundamental invariant theory (\cite{procesi:invariants}). 
This is all contained in Section~\ref{sec:SpecialFiber}.

Next we let $\cX\to\mathbb{D}$ be a family of polarized varieties over a pointed smooth curve with the following properties: the fibers $\cX_t$ for $t\not=0$ are {\rm HK} manifolds of  {\rm K3}$^{[n]}$-type, the restriction to $\cX_t$ of the polarization of $\cX$ is a polarization $\lambda_t$ of square $2$ and divisibility $2$, and the fiber $\cX_0$ is isomorhic to the Donaldson--Uhlenbeck--Yau moduli space $\ov{M}_{S,n}$. 
Moreover we require that there exists a  fiberwise involution of $\cX$ which when restricted to $\cX_t$ for $t\not=0$ is equal to $\tau_{\lambda_t}$, and when restricted to $\cX_0=\ov{M}_{S,n}$ is equal to the involution described above. 
We let $\cY_{\pm}\subset\cX$ be the closure of the variety swept out by $\Fix(\tau_{\lambda_t})_{\pm}$ for $t\in(\mathbb{D}\setminus\{0\})$.
Then $\cY_{\pm}$ is  integral and $\cY_{\pm}\to\mathbb{D}$ is flat. 
The main result of Section~\ref{sec:linearization} is that the (scheme theoretic) fiber over $0$ of $\cY_{-}\to\mathbb{D}$ is equal to $\ov{M}_{\PP^2,n}$ (embedded in $\cX_0=\ov{M}_{S,n}$ via the pullback $f^*$ as discussed above).

In Section~\ref{sec:MainResult} we prove Theorem~\ref{thm:main1} by considering the family $\cY_{-}\to\mathbb{D}$ defined above. 
We show that $\cY_{-}$ is normal, Gorenstein, and that  the fiber over $0$, i.e., $\ov{M}_{\PP^2,n}$, has dualizing line bundle isomorphic to the restriction of $(\cL^{\vee})^{\otimes  3}$, where $\cL$ is the relative ample line bundle on $\cX\to\mathbb{D}$.
Since the restriction of $\cL$ to $\ov{M}_{\PP^2,n}$ is a generator of the Picard group, this completes the proof of Theorem~\ref{thm:main1}. 

The proof of Theorem~\ref{thm:main2} is similar, with $\cY_{+}\to\mathbb{D}$ replacing $\cY_{-}\to\mathbb{D}$. 
We restrict to $n=4$ because  in this case $\cX_0\cong \ov{M}_{S,4}$ is birational to $S^{[4]}$ and we can explicitly describe all components of the fixed locus of the involution by relating them to the fixed components of the corresponding involution of $S^{[4]}$.

\subsection*{Acknowledgements}
The paper benefited from many useful discussions with the following colleagues who we gratefully acknowledge: Enrico Arbarello, Alessio Bottini, Olivier Debarre, Tommaso de Fernex, Enrico Fatighenti, Franco Giovenzana, Luca Giovenzana, Daniel Huybrechts, Alexander Kuznetsov, Manfred Lehn, Marco Manetti, Francesco Meazzini, Giovanni Mongardi, Alexander Perry, Laura Pertusi, Christoph Sorger, Paolo Stellari, Claire Voisin, Chenyang Xu, Xiaolei Zhao.
We also thank the referees for the detailed reports, which improved the exposition of the paper and fixed several inaccuracies.


\section{The special fiber}\label{sec:SpecialFiber}

In this section we study the Donaldson--Uhlenbeck--Yau compactification for rank-2 torsion-free sheaves, both in the case of a very general K3 surface $S$ of genus~2 and in the case of the projective plane. The main result (Theorem~\ref{thm:embedding}) is that the pull-back morphism induces a closed embedding of the appropriate moduli spaces. To prove this we give a moduli-theoretic description of the two moduli spaces in terms of Bridgeland moduli spaces of semistable complexes (respectively, in Section~\ref{subsec:DUYK3} and Section~\ref{subsec:DUYP2}); this allows us to give a local analytic description of the singularities of these moduli spaces.

\subsection{The Donaldson--Uhlenbeck--Yau compactification for K3 surfaces}\label{subsec:DUYK3}

Let $(S,h)$ be a polarized K3 surface such that $\NS(S)=\ZZ h$, $h^2=2d$. Let $v=(r,l,s)\in H^0(S;\ZZ)\oplus\NS(S)\oplus H^4(S;\ZZ)$ be a primitive Mukai vector with $r\ge 0$, and let $M(v)$ be the moduli space of semistable sheaves on the polarized $K3$ surface $(S,h)$ with Mukai vector $v$. 
Then the variety $M(v)$ is a projective HK manifold of $\mathrm{K3}^{[m]}$-type, where $m=\frac{v^2}{2}+1$. 
If $m\geq2$, then the Mukai map (see, e.g.,~\cite[Main Theorem]{Kieran:weight2} or~\cite[Eqn.~(1.6)]{Yos:moduli})
\[
\theta_v\colon v^\perp \xlongrightarrow{\cong} H^2(M(v),\ZZ)
\]
gives an isometry of weight-2 Hodge structures; here $v^\perp\subset H^*(S,\ZZ)$ has the induced weight-2 Hodge structure induced by the Mukai Hodge structure on $H^*(S,\ZZ)$. The map $\theta$ is determined up to sign. We decide to choose the opposite of the convention adopted in \emph{loc.~cit.}
More precisely, we let $p\colon S\times M(v)\lra S$ and $q\colon S\times M(v)\lra M(v)$ be the projections. Let $\cF$ be a universal sheaf on $S\times M(v)$ or, if it does not exist, a  quasi-universal sheaf of similarity $\sigma\in\NN_{+}$, i.e., such that for all $[\cE]\in M(v)$ the restriction of $\cF$ to $S\times\{[\cE]\}$ is isomorphic to $\cE^{\oplus\sigma}$. By~\cite[Thm~A.5]{Mukai:Tata}
a quasi-universal sheaf exists. We let
\begin{equation}\label{mappatheta}
\begin{matrix}
v^{\bot} & \overset{\theta_v}{\lra} & H^2(M(v)) \\
\alpha & \longmapsto & q_{*}\left(\frac{1}{\sigma}\ch(\cF)\cdot p^{*}(\alpha^{\vee}\cdot\sqrt{\td_S})\right)_6,
\end{matrix}
\end{equation}
where, if $\alpha^{2k}\in H^{2k}(S)$ is the degree-$2k$ component of $\alpha$, we let
$\alpha^{\vee}\coloneqq \alpha^0-\alpha^2+\alpha^4$, and the subscript $6$ means that we consider the component in $H^6(S\times M(v))$. 

\begin{rema}\label{rema:dondivisors}
Let $v=(r,l,s)$. 
If $\beta\in H^2(S)$ then 
$(0,\beta,(l\cdot \beta)/r)\in v^{\bot}$, and we have
\begin{equation*}
\mu_v(\beta)=\theta_v(0,\beta,(l\cdot \beta)/r)
\end{equation*}
where $\mu_v$ is Donaldson's map (see~\cite[Sect.~3]{Kieran:weight2}). Moreover, if $\cL_1$ is the determinantal line bundle on $M(v)$ defined in~\cite{LePotier:DetLineBundle} (see also~\cite[Sect.~8.1]{HL:moduli}), we have
\begin{equation*}
c_1(\cL_1)=\theta_v(0,rh,(l\cdot h)).
\end{equation*}
\end{rema}

Now let $n\geq2$ be an integer such that $d-n+1\equiv 0 \pmod 2$.
We consider the primitive Mukai vector
\begin{equation*}
v_n:=\left(2,-h,\frac{d-n+1}{2}\right)\in H^{*}_\mathrm{alg}(S,\ZZ).
\end{equation*}
and the moduli space $M_{S,n}:=M(v_n)$ of dimension~$2n$.
Let us consider the algebraic Mukai vectors $(0,h,-d),(2,-h,\frac{d+n-1}{2})\in v_n^\perp$ and let
\begin{equation}\label{lamdel}
\lambda_S := \theta\bigg(0,h,-d\bigg) \qquad 
\delta:=\theta\left(2,-h,\frac{d+n-1}{2}\right)
\end{equation}
in $\mathrm{NS}(M_{S,n})$.
We have the following well-known result.

\begin{lemm}\label{lem:DUYdivisor}
In the above notation, we have:
\begin{enumerate}[{\rm(a)}]
\item\label{enum:DUYdivisor1} $\lambda_S\in\NS(M_{S,n})$ is a semiample divisor class with $q(\lambda_S)=2d$ which induces a divisorial contraction
\[
\phi\colon M_{S,n} \lra \overline{M}_{S,n},
\]
where $\overline{M}_{S,n}$ is the Donaldson--Uhlenbeck--Yau compactification of the moduli space of slope stable locally free sheaves:
\[
\overline{M}_{S,n}:=\Proj\left(\bigoplus_{k=0}^{\infty} H^0(M_{S,n},\cO_{M_{S,n}}(\lambda_S)^{\otimes k})\right).
\]
The variety $\overline{M}_{S,n}$ is normal, $\QQ$-factorial, with canonical Gorenstein singularities and trivial canonical bundle.
\item\label{enum:DUYdivisor2} $\delta\in\NS(M_{S,n})$ is the class of the irreducible divisor
\[
\Delta:=\left\{[\cF]\in M_{S,n} \,:\, \text{$\cF$ is not locally free}\right\}
\]
which is the exceptional locus of the contraction $\phi$.
\item\label{enum:DUYdivisor3} $\Delta$ has a stratification by locally-closed subsets
\[
\{ 0 \} \subset \Delta_{n,\lfloor n/2\rfloor} \subset \ldots \subset \Delta_{n,1} = \Delta, 
\]
indexed by the length of the quotients $\cF^{\vee\vee}/\cF$.
\end{enumerate}
\end{lemm}

\begin{proof}
To show~\ref{enum:DUYdivisor1}, by Remark~\ref{rema:dondivisors} we have $c_1(\cL_1)=2\lambda_S$, where $\cL_1$ is as in the remark. 
Let $\cD$ be the determinant line bundle on $M_{S,n}$ defined in~\cite{Li:uhlencomp} (and denoted by $\cL_{\cM,1}$ in \emph{loc.~cit.}).
Then $\cL_1\cong \cD^{\otimes 2}$, see~\cite[Lemma~8.3.4]{HL:moduli}. Hence the first sentence of~\ref{enum:DUYdivisor1} follows from the main result of~\cite{Li:uhlencomp}. 
In order to prove  the second sentence, we observe that by its definition, $\overline{M}_{S,n}$ is normal.
Since the contraction $\varphi$ is divisorial of relative Picard number~1, by using the restriction sequence for Weil divisors on the smooth locus, we deduce that the base is $\Q$-factorial.
Moreover, it has symplectic singularities, which are thus canonical.

Finally,~\ref{enum:DUYdivisor2} Follows from a straightforward computation and~\ref{enum:DUYdivisor3} is clear.
\end{proof}

Explicitly (see~\cite{Li:uhlencomp}), as a set the (closed) points of $\overline{M}_{S,n}$ are pairs $([\cE],Z)$ where 
$\cE$ is an $h$-slope stable vector bundle with Mukai vector $v(\cE)=v_n+\ell\cdot(0,0,1)$ for some $0\le \ell\le \lfloor n/2\rfloor$, and $[Z]\in S^{(\ell)}$. 
The contraction map $\phi\colon M_{S,n}\to \overline{M}_{S,n}$ is described as follows.
Let $[\cF]\in M_{S,n}$, and let
\[
0\lra \cF\lra \cF^{\vee\vee}\lra \cQ\lra 0
\]
be the natural exact sequence, where $\cF^{\vee\vee}$ is the double dual of $\cF$ (an $h$-slope stable locally free sheaf).
The sheaf $\cQ$ is Artinian, let $\ell$ be its  length.
Then $[\cF]\in\Delta_{\ell}$ and
\[
\varphi([\cF])=\left([\cF^{\vee\vee}],\sum_{p}\ell(\cQ_p)p\right).
\]

\begin{rema}\label{rmk:divisibility}
The divisibility $\div(\lambda_S)$ of $\lambda_S$ is either~$1$ or~$2$, and it equals $2$ if and only if $d+n-1\equiv 0 \pmod 4$.
Indeed this follows from the fact that a vector $(r,\alpha,s)\in H^*(S,\Z)$ belongs to $v_n^\perp$ if and only if $(\alpha,h)=-r\frac{d-n+1}{2}-2s$.
\end{rema}

We want to understand the local structure of $\overline{M}_{S,n}$ at a singular point, namely when $\ell\neq0$ in the above description.
To this end we want to reinterpret $\overline{M}_{S,n}$ as a moduli space of Bridgeland semistable objects in the bounded derived category of coherent sheaves $\Db(S)$.

For $\alpha,\beta\in\RR$, $\alpha>0$, we consider Bridgeland stability conditions on $\Db(S)$ of the form $\sigma_{\alpha,\beta}=(Z_{\alpha,\beta},\coh^\beta(S))$ (see, e.g., \cite[Example 3.7]{involutions1}). We further consider the moduli space $\overline{M}_{S,n}':=M_{S,\sigma_{\alpha,-1/2}}(-v_n)$ of Bridgeland semistable objects with phase~1 and Mukai vector $-v_n$ with respect to a stability condition $\sigma_{\alpha,-1/2}$, for an arbitrary $\alpha\gg 0$.
This moduli space has been constructed in general as a good moduli space in the sense of Alper (see~\cite{Alper:GoodModuli,AHLH:Moduli} and \cite[Theorem 21.24]{BLMNPS:family}); the special case for our vector $v_n$ and $\beta=-1/2$ has been studied in~\cite{Tajakka:Uhlenbeck}.
In particular, by \cite[Theorem 1.1]{Tajakka:Uhlenbeck}, $\overline{M}_{S,n}'$ is a projective variety of dimension~$2n$ and we have a natural proper birational bijective morphism
\[
g_n\colon \overline{M}_{S,n} \longrightarrow \overline{M}_{S,n}',
\]
defined on objects by
\begin{equation}\label{eq:GiesekerToDUY}
g_n\left([\cE],\sum_p \ell_p\, p\right):=\left[\cE[1]\oplus \bigoplus_p\, k(p)^{\oplus \ell_p}\right],
\end{equation}
where $k(p)$ denotes the skyscraper sheaf at the point $p$ (we will adhere to this convention throughout). 

\begin{prop}\label{prop:UhlenbeckVsBridgeland}
In the above notation, the bijective morphism $g_n$ is an isomorphism.
\end{prop}

Before proving Proposition \ref{prop:UhlenbeckVsBridgeland}, we remark that by Lemma~\ref{lem:DUYdivisor}\ref{enum:DUYdivisor1}, we know that $\overline{M}_{S,n}$ is normal. 
Hence, since $g_n$ is proper, birational and bijective, in order to prove Proposition~\ref{prop:UhlenbeckVsBridgeland} we only need to show that $\overline{M}_{S,n}'$ is normal too.
To this end, before presenting the proof, we recall explicitly the local structure of $\overline{M}_{S,n}'$. This will also be used in our analysis later.

Let $F:=\cE[1]\oplus \bigoplus_{j=1}^m\, k(p_j)\otimes W_j$ be a polystable element in $\overline{M}_{S,n}'$, where $p_1,\dots,p_m\in S$ are distinct closed points and $W_j$ are vector spaces of dimension $a_j>0$, with $\ell:=\sum_j a_j$.
We let $V_{S,F}$ be the vector space
\[
V_{S,F}:=\Ext^1_S(F,F).
\]
Explicitly, 
\begin{equation*}\label{eq:SplittingK3Antony}
V_{S,F}=\Ext^1_S(\cE,\cE)\oplus V_{S,F}'    
\end{equation*}
where
\begin{equation}\label{eqroma24}
\begin{split}
V_{S,F}':=\bigoplus_{j=1}^m\Big( (\Hom_S(\cE,k(p_j))\otimes W_j) &\oplus (\Ext^2_S(k(p_j),\cE)\otimes W_j^\vee)\\
&\oplus (\Ext^1_S(k(p_j),k(p_j))\otimes \End(W_j)) \Big).
\end{split}
\end{equation}
The group
\begin{equation}\label{eq:group}
G_{S,F}:=\left(\CC^*\times \prod_{j=1}^m\mathrm{GL}(W_j)\right)\bigg/\CC^*
\end{equation}
acts on $V_{S,F}$ by conjugation on the direct sum.
Finally, consider the moment map for this action
\[
\mu_{S,F}\colon V_{S,F} \longrightarrow \Ext^2_S(F,F)_0,
\]
which is given by the Yoneda product, where $\Ext^2_S(F,F)_0$ denotes the traceless part of $\Ext^2_S(F,F)$.

The deformation theory of complexes behaves exactly as for sheaves (see~\cite[Theorem 3.1.1]{Lieblich:moduli}): hence, the base space of the formal semiuniversal deformation of $F$ is the scheme-theoretic fibre of the Kuranishi map
\[
\kappa_{S,F} \colon \widehat{V}_{S,F} \longrightarrow \Ext^2_S(F,F)_0,
\]
where $\widehat{V}_{S,F}$ denotes the formal completion of $V_{S,F}$ at $0$ (see, e.g., \cite[Section 3]{AS:singularities}).
As observed in~\cite[Section 4]{AS:singularities}, the Kuranishi map can be chosen to be $G_{S,F}$-equivariant.
Moreover, the quadratic term of the Kuranishi map is exactly the map $\mu_{S,F}$.

\begin{lemm}\label{lem:KuranishiK3}
We can choose a Kuranishi map so that the equality $\kappa_{S,F}=\mu_{S,F}$ holds. In particular, we have an isomorphism of analytic germs
\[
\left(\overline{M}_{S,n}',[F]\right)\cong \Big( \mu_{S,F}^{-1}(0)\sslash G_{S,F},[0]\Big).
\]
\end{lemm}

Lemma~\ref{lem:KuranishiK3} is proved in~\cite[Section~3.2]{PZ:FormalityBridgeland} in general, for all moduli spaces of Bridgeland semistable objects on the derived category of a K3 surface, by using the approach in~\cite{BMM:Formality} (the case of a generic stability condition, which is not sufficient for our purposes, was proven earlier in~\cite{BZ:Kuranishi}).
We give here a quick sketch of the proof, by following~\cite{PZ:FormalityBridgeland}. For a complete argument see also~\cite[Theorem 3.2]{AS:Formality}.

\begin{proof}[Proof of Lemma~\ref{lem:KuranishiK3}.]
Let us denote by $\cF'$ the coherent sheaf $\oplus_{j=1}^m\, k(p_j)\otimes W_j$ and let us consider the automorphism group $\mathrm{Aut}(F)$.
This splits as
\begin{equation}\label{eq:SplittingAutGroup}
\mathrm{Aut}(F) = \mathrm{Aut}(\cE[1]) \times \mathrm{Aut}(\cF').
\end{equation}
First of all, by using the above splitting~\eqref{eq:SplittingAutGroup}, we observe that we can construct an $\mathrm{Aut}(F)$-equivariant locally free resolution
\[
\cG^\bullet=\{0\to \cG^{-2} \to \cG^{-1} \to \cG^0 \to 0 \}
\]
of $F$ of length~2, by taking the direct sum of $\cE[1]$ with an $\mathrm{Aut}(\cF')$-equivariant locally free resolution of $\cF'$.

The next step is to consider, as in~\cite[Section 5]{BMM:Formality}, the Dolbeault DG-Lie algebra presentation of $\mathrm{RHom}_S(F,F)$ given by
\[
K:=\bigoplus_{q,r,s} A^{0,q}\big( \mathcal{H}om_S(\cG^r,\cG^s) \big),
\]
where $A^{0,q}(\mathcal{H}om(\cG^r,\cG^s))$ is the space of $\mathcal{H}om_S(\cG^r,\cG^s)$-valued global $(0,q)$-forms.
The results in~\cite[Section 3]{BMM:Quadraticity} hold without changes in our context, given the splitting~\eqref{eq:SplittingAutGroup}; in particular, we have that $K$ admits an $\mathrm{Aut}(F)$-action, such that each degree of $K$ is a rational representation of $\mathrm{Aut}(F)$. 

Then the proof in~\cite[Theorem 5.1]{BMM:Formality} directly applies, showing that the DG-Lie algebra $\mathrm{RHom}_S(F,F)$ is formal.
The results in~\cite{FIM:dgla} for coherent sheaves extend without changes to complexes. Hence, the deformation theory of $F$ is controlled by the deformation theory of the DG-Lie algebra $\mathrm{RHom}_S(F,F)$; in particular, by formality, it follows that the Kuranishi map can be chosen to be quadratic, as we wanted.
\end{proof}

We can describe $\mu^{-1}_{S,F}(0)\sslash G_{S,F}$ in terms of quivers, as reviewed in~\cite[Section 4]{AS:Formality}.
Indeed, given a quiver $Q$, we can associate a new quiver, called the double quiver associated to $Q$ and denoted by $\overline{Q}$, whose vertex set is the same as that of $Q$ and whose arrow set is obtained from that of $Q$ by adding, for each arrow, a new arrow in the opposite direction.
By fixing a dimension vector, the double quiver then naturally has a moment map and the Marsden--Weinstein reduction (in the sense of \cite[Section 1]{CB:normality}) is nothing but the GIT quotient of the zero-locus of the moment map by the associated group of automorphisms.

In our case, we consider the double quiver $\overline{Q}$ of the following quiver
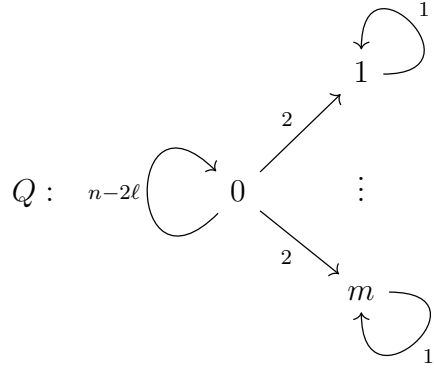
\[
\begin{tikzcd}
&&& 1 \arrow[out=0,in=90,loop,swap,"1"]\\
Q:&& 0 \arrow[out=-135,in=135,loop,"n-2\ell"]\arrow[ur,"2"]\arrow[dr,swap,"2"]&\vdots\\
&&& m \arrow[out=0,in=-90,loop,"1"]
\end{tikzcd}
\]
with vertex set $\{0,1,\dots,m\}$, two arrows from the vertex~0 to each of the vertices $1,\dots,m$, one loop around each vertex $1,\dots,m$ and $(n-2\ell)$~loops around the vertex~0.
Here we look at the dimension vector
\[
\alpha:=(1,a_1,\dots,a_m).
\]

\begin{proof}[Proof of Proposition~\ref{prop:UhlenbeckVsBridgeland}.]
By~\cite[Theorem 1.1]{CB:normality}, the quotient $\mu^{-1}_{S,F}(0)\sslash G_{S,F}$ with the reduced induced scheme structure is normal.
We only need to show it is reduced.
To this end, it is enough to show that $\mu^{-1}_{S,F}(0)$ is integral. To prove this, we use \cite{CB:geometry}.
We recall some notation from~\emph{loc.~cit.}.
For a dimension vector $\beta=(b_0,b_1,\dots,b_m)$, we set
\[
p(\beta):=1+\sum_{u\in Q} b_{h(u)}b_{t(u)} - \beta^2,
\]
where for an arrow $u\in Q$, we denote by $h(u)$, respectively $t(u)$ the head vertex, respectively the tail vertex and we set
\[
\beta^2=\sum_{i=0}^m b_i^2.
\]

For us, we need to consider two types of vectors:
\begin{equation}\label{eq:CBfunctionp}
\begin{split}
    &p((1,b_1,\dots,b_m))=n-2\ell + 2(b_1+\dots+b_m)\\
    &p((0,b_1,\dots,b_m))=1.
\end{split}
\end{equation}
Then, by~\emph{loc.~cit.}, the fiber $\mu^{-1}_{S,F}(0)$ is a reduced and irreducible complete intersection of dimension
\[
2n+\sum_{i=1}^ma_i^2=\alpha^2-1+2p(\alpha)
\]
if the following holds.
Let us write
\[
\alpha=(1,a_1,\dots,a_m)=\sum_{j=0}^r \beta^{(j)},
\]
where
\[
\beta^{(0)}:=(1,b_1^{(0)},\dots,b_m^{(0)}),\qquad  \beta^{(j)}:=(0,b_1^{(j)},\dots,b_m^{(j)}),\text{ for }j=1,\dots,r,
\]
with $b_i^{(0)},b_i^{(j)}\geq0$, $(b_1^{(j)},\dots,b_m^{(j)})\neq(0,\dots,0)$, for all $i=1,\dots,m$, $j=1,\dots,r$, $r\geq 1$.
By~\cite[Theorem~1.2(2)]{CB:geometry}, we need then to show that
\[
p(\alpha)>\sum_{l=0}^rp(\beta^{(l)}).
\]
To this end, we use~\eqref{eq:CBfunctionp} and $r\geq1$, and we have:
\[
\begin{split}
\sum_{l=0}^rp(\beta^{(l)}) & = (n-2\ell)+2\sum_{i=1}^m b_i^{(0)} + r \\
& < (n-2\ell)+2\sum_{i=1}^m (b_i^{(0)} + r) \\
& \leq (n-2\ell)+2\sum_{i=1}^m \sum_{j=0}^r b_i^{(j)}\\
& = (n-2\ell)+2\sum_{i=1}^m a_i = p(\alpha),
\end{split}
\]
as we wanted.
\end{proof}

\begin{rema}\label{rmk:SplitKuranishi}
We notice that since $\cE$ is simple, the $\CC^*$-factor of the group  $G_{S,F}$ acts trivially on $\Ext^1_S(\cE,\cE)$, and hence the the group  $G_{S,F}$ acts trivially on the factor $\Ext^1_S(\cE,\cE)$ of  $V_{S,F}$.
Moreover, a computation shows that the factor $\Ext^1_S(\cE,\cE)$ of $V_{S,F}$ is contained in the radical of the Yoneda product.
Therefore, from Lemma~\ref{lem:KuranishiK3} we deduce a splitting
\begin{equation*}\label{eq:SplittingK3}
\left(\overline{M}_{S,n}',[F]\right)\cong \Big(\Ext^1_S(\cE,\cE)\times   \big((\mu_{S,F}')^{-1}(0)\sslash G_{S,F}\big),[0]\Big),
\end{equation*}
where
\[
\mu_{S,F}'\colon V_{S,F}'\to \bigoplus_{j=1}^m \Big( \Ext^2_S(k(p_j),k(p_j))\otimes \End(W_j) \Big)
\]
is induced by $\mu_{S,F}$ by restriction to the subspace $V'_{S,F}$ defined in~\eqref{eqroma24}.
\end{rema}

\begin{rema}\label{rmk:AmpleLineBundleDUYK3}
By~\cite[Theorem 21.25]{BLMNPS:family}, the class $\lambda_S$ on $M_{S,n}$ descends to the class $\overline{\lambda}'$ of a primitive ample Cartier divisor on $\overline{M}_{S,n}'$.
In particular, $\overline{\lambda}_S:=g_n^*(\overline{\lambda}_S')$ is the class of a primitive ample Cartier divisor on $\overline{M}_{S,n}$; thus, $\NS(\overline{M}_{S,n})=\ZZ \cdot \overline{\lambda}_S$.
\end{rema} 

\begin{rema}\label{rmk:KuranishiDelta1}
At a point $F:=\cE[1] \oplus k(p)$ in $\phi(\Delta_{n,1})$, for $p\in S$, the Kuranishi morphism is
\[
\begin{matrix}
     \Ext^1_S(\cE,\cE) \oplus U \oplus U^\vee \oplus \Ext^1_S(k(p),k(p)) & \overset{\mu_{S,F}}{\longrightarrow} & \Ext^2(F,F)_0=\CC\\
    (\eta, a_1,a_2,b_1,b_2, \nu) & \longmapsto & a_1b_1+a_2b_2.
 \end{matrix}
\]
where $U=\Ext^1(\cE[1],k(p))$, $U^\vee=\Ext^1(k(p),\cE[1])$. Recall that the Serre duality pairing is given by the Yoneda product, so choosing a basis for $U$ and the dual basis for $U^\vee$, the Kuranishi morphism is as above.
The action of $G_{S,F}\cong\CC^*$ is given as
\[
\gamma.(\eta,a_1,a_2,b_1,b_2,\nu)=(\eta,\gamma a_1,\gamma a_2,\gamma^{-1} b_1, \gamma^{-1} b_2,\nu).
\]
Hence, the non-trivial part of the invariant ring is given by
\[
\CC[a_1,a_2,b_1,b_2]^{\CC^*} \cong \CC[u_1,u_2,u_3,u_4],
\]
where $u_1=a_1b_1$, $u_2=-a_1b_2$, $u_3=a_2b_1$, $u_4=a_2b_2$  with relation $u_1u_4+u_2u_3=0$ (the moment map becomes linear on the GIT quotient: $\mu_{S,F}=u_1+u_4$).

This shows that $\overline{M}_{S,n}$, locally at the point $[F]$, is given as a product
\[
\Ext^1_S(\cE,\cE)\times Q \times \Ext^1_S(k(p),k(p)),
\]
where $Q=\Spec(\CC[u_1,u_2,u_3]/(u_1^2-u_2u_3))$ is the quadric cone.
In particular, we see explicitly that $\overline{M}_{S,n}$ is a local complete intersection along $\Delta_{n,1}$.
\end{rema}

\subsection{The Donaldson--Uhlenbeck--Yau compactification for the projective plane}\label{subsec:DUYP2}

In this section, we recall the analogue of  Section~\ref{subsec:DUYK3} for moduli spaces on $\PP^2$. 
Let $n\equiv 0 \pmod 4$.
Fix the Chern character
\[
u_n:=\left(2, -1, -\frac{n+2}{4} \right)\in H^*(\PP^2,\QQ).
\]
We consider the moduli space $M_{\PP^2,n}$ of Gieseker stable sheaves on $\PP^2$ with Chern character $u_n$.
It is a smooth projective variety of dimension $n$. According to~\cite[Theorem~8.3.3]{HL:moduli} one has  $\omega_{M_{\PP^2,n}}^\vee = \cO_{M_{\PP^2,n}}(3\lambda_{\PP^2})$, where $\lambda_{\PP^2}$ is the determinant line bundle associated to the class of a line in $\PP^2$. By~\cite[Lemma~8.3.4]{HL:moduli} we have $\lambda_{\PP^2}\cong \cD^{\otimes 2}$, where $\cD$ is the determinant line bundle on $M_{\PP^2,n}$ defined in~\cite{Li:uhlencomp} (and denoted by $\cL_{\cM,1}$ in \emph{loc.~cit.}). 
By the main result of~\cite{Li:uhlencomp} it follows that $M_{\PP^2,n}$ is log-Fano and thus a Mori Dream Space (see, e.g.,~\cite[Example 3.2]{Castravet:MDS}).

The Donaldson--Uhlenbeck--Yau compactification $\phi_{\PP^2}\colon M_{\PP^2,n}\to \overline{M}_{\PP^2,n}$ is the divisorial contraction induced by $\lambda_{\PP^2}$ (or equivalently $\omega_{M_{\PP^2,n}}^\vee$).
Thus, the variety $\overline{M}_{\PP^2,n}$ is normal, $\QQ$-factorial, with canonical singularities.

As in the K3 surface case, we also have a description of $\overline{M}_{\PP^2,n}$ as a Bridgeland moduli space in $\Db(\PP^2)$: this is well-known but we sketch the argument here for completeness.
We consider Bridgeland stability conditions $\sigma_{\alpha,\beta}$ on $\Db(\PP^2)$, for $\alpha>0$, as in~\cite[Theorem 6.10]{MS:lectures} and the moduli space $\overline{M}_{\PP^2,n}':=M_{\PP^2,\sigma_{\alpha,-1/2}}(-u_n)$ of Bridgeland semistable objects with phase~1 and Chern character $-u_n$.
We again have a morphism
\[
g_{\PP^2,n}\colon \overline{M}_{\PP^2,n} \lra \overline{M}_{\PP^2,n}'
\]
which is bijective and defined on objects as in~\eqref{eq:GiesekerToDUY}.

\begin{prop}\label{prop:UhlenbeckVsBridgelandP2}
In the above notation, the bijective morphism $g_{\PP^2,n}$ is an isomorphism.
\end{prop}

Proposition~\ref{prop:UhlenbeckVsBridgelandP2} is well-known, since both $M_{\PP^2,n}$ and $\overline{M}_{\PP^2,n}'$ can be described globally as moduli spaces of quiver representations in the sense of King (\cite{King:Moduli}) by considering the quiver with relations associated to the exceptional collection
\[
\left\{\cO_{\PP^2}(-1), \Omega_{\PP^2}(1), \cO_{\PP^2} \right\}
\]
(see~\cite{DLP}; in the quiver language, see, e.g.,~\cite[Main Theorem 1.3]{Okhawa:P2})).
Then we can directly show that $g_{\PP^2,n}$ is an isomorphism, by looking at global sections of line bundles. 

We will sketch instead a proof as in the K3 case, by showing first that the Kuranishi map is again quadratic, since we will need this description in the next section.
We let 
\[
F:=\cE[1]\oplus \bigoplus_{j=1}^m\, k(p_j)\otimes W_j
\]
be a polystable element in $\overline{M}_{\PP^2,n}'$, where $p_1,\dots,p_m\in \PP^2$ are distinct closed points and $W_j$ are vector spaces of dimension $a_j>0$, with $\ell:=\sum_j a_j$.
We let $V_{\PP^2,F}:=\Ext^1_{\PP^2}(F,F)$ and the group 
\begin{equation}\label{eq:groupP^2}
G_{\mathbb{P}^2,F}:=\left(\CC^*\times \prod_{j=1}^m\mathrm{GL}(W_j)\right)\bigg/\CC^*
\end{equation}
acts on $V_{\mathbb{P}^2,F}$ by conjugation component-wise.

Consider the Yoneda product
\[
\mu_{\PP^2,F}\colon V_{\PP^2,F} \lra \Ext^2_{\PP^2}(F,F)
\]
and the Kuranishi map
\[
\kappa_{\PP^2,F}\colon \widehat{V}_{\PP^2,F} \lra \Ext^2_{\PP^2}(F,F).
\]

\begin{lemm}\label{lem:KuranishiP2}
We can choose a Kuranishi map so that the equality $\kappa_{\PP^2,F}=\mu_{\PP^2,F}$ holds. In particular, we have an isomorphism of analytic germs
\[
\left(\overline{M}_{\PP^2,n}',[F]\right)\cong \Big( \mu_{\PP^2,F}^{-1}(0)\sslash G_{\PP^2,F},[0]\Big).
\]
\end{lemm}

\begin{proof}
The DG-Lie algebra $\mathrm{RHom}_{\PP^2}(F,F)$ controls the local structure of $\overline{M}_{\PP^2,n}$.
We claim that the former is formal. 
This can be checked either directly, or by reducing to the case of K3 surfaces; we sketch the latter argument, by using the ideas in \cite{BZ:Kuranishi} and \cite{BMM:Formality}.

Let us choose a smooth very general sextic curve $\mathsf{\Gamma}$ in $\PP^2$ which does not pass through the points $p_1,\dots,p_m$. Then we consider the 2-1 cover $f\colon S\to \PP^2$ ramified at $\mathsf{\Gamma}$; the assumption on $\mathsf{\Gamma}$ guarantees that the pull-back complex $f^*F$ is again polystable in $S$.
Let us further consider the factorization of $f$ via the stack $\left[S/\mu_2\right]$:
\[
f\colon S \xlongrightarrow{f_1} \left[S/\mu_2\right] \xlongrightarrow{f_2} \PP^2.
\]
We will examine the two morphisms $f_1$ and $f_2$ separately.
We start with $f_2$.
The derived category of $\left[S/\mu_2\right]$ is given by the $\mu_2$-equivariant derived category $\mathrm{D}^{\mathrm{b}}_{\mu_2}(S)$.
By~\cite[Proposition 6.10]{CanonacoStellari:Uniqueness}, this category has a strongly unique DG-enhancement (see, for example, \cite{CanonacoStellari:Tour} for a survey on DG-enhancements and their uniqueness).
By~\cite{PolVdB:equivariant}, the pull-back functor $f_2^*$ realizes $\Db(\PP^2)$ as a full semiorthogonal component in $\mathrm{D}^{\mathrm{b}}_{\mu_2}(S)$; in particular, it inherits a DG enhancement from an enhancement of $\mathrm{D}^{\mathrm{b}}_{\mu_2}(S)$.
Since $\Db(\PP^2)$ has also a strongly unique DG-enhancement (see~\cite{LO:Uniqueness}), we can apply~\cite[Proposition 2.13]{BZ:Kuranishi} and deduce that the DG-Lie algebra $\mathrm{RHom}_{\PP^2}(F,F)$ is formal if and only if the DG-Lie algebra $\mathrm{RHom}_{[S/\mu_2]}(f_2^*F,f_2^*F)$ is formal.

Now we can argue exactly as in~\cite{BMM:Formality} for the morphism $f_1$.
Indeed, this is \'etale and we can consider the algebra $\cC:=f_{1,*}\cO_S$ of rank~2.
Then, by proceeding as in the proof of~\cite[Theorem 5.3]{BMM:Formality}, we deduce the formality of $\mathrm{RHom}_{[S/\mu_2]}((f_2^*F)\otimes\cC,(f_2^*F)\otimes\cC)$.
By the formality transfer theorem~\cite[Theorem 3.4]{Manetti:FormalityCriteria}, this implies the formality of $\mathrm{RHom}_{[S/\mu_2]}(f_2^*F,f_2^*F)$, as we needed.
\end{proof}

To prove Proposition~\ref{prop:UhlenbeckVsBridgelandP2}, we starty by observing that, as in the K3 surface case, neither the Yoneda product nor the group $G_{\P^2,F}$ involve the first direct factor in $V_{\PP^2,F}$.
More precisely, if we write, as in the K3 case
\[
V_{\PP^2,F} = \Ext^1_{\PP^2}(\cE,\cE) \oplus V_{\PP^2,F}',
\]
then we have a splitting as in Remark~\ref{rmk:SplitKuranishi}:
\begin{equation*}\label{eq:SplittingP2}
\left(\overline{M}_{\PP^2,n}',[F]\right)\cong \Big(\Ext^1_{\PP^2}(\cE,\cE) \times \big( (\mu_{\PP^2,F}')^{-1}(0)\sslash G_{\P^2,F}\big),[0]\Big),
\end{equation*}
where
\[
\mu_{\PP^2,F}'\colon V_{\PP^2,F}'\lra \bigoplus_{j=1}^m \Big( \Ext^2_{\PP^2}(k(p_j),k(p_j))\otimes \End(W_j) \Big)
\]
is induced by $\mu_{\PP^2,F}$ by restriction.
The elementary observation is that the pair $(V_{\PP^2,F}',\mu_{\PP^2,F}')$ and the group $G_{\P^2,F}$ only depend on the vector spaces $W_j$ and the analytic neighbours of the points $p_j$, and not on the global structure of $\cE$.
Hence, the singularities of $\overline{M}_{\PP^2,n}'$ are analytically locally the same as those of $\overline{M}_{S,n/2}'$.
In particular, $\overline{M}_{\PP^2,n}'$ is normal, thus completing the proof of Proposition~\ref{prop:UhlenbeckVsBridgelandP2}.

\begin{rema}\label{rmk:AmpleLineBundleDUYP2}
As in Remark~\ref{rmk:AmpleLineBundleDUYK3}, the class $\lambda_{\PP^2}$ descends to the class $\overline{\lambda}_{\PP^2}$ of a primitive ample Cartier divisor on $\overline{M}_{\PP^2,n}$.
In particular, $\overline{M}_{\PP^2,n}$ has Gorenstein singularities and is Fano of index~3.
The fact that the singularities are rational follows for example from the above observation that $\overline{M}_{\PP^2,n}$ is analytically locally isomorphic to the moduli space $\overline{M}_{S,n/2}$.
Moreover, as in Remark~\ref{rmk:KuranishiDelta1}, $\overline{M}_{\PP^2,n}$ is a local complete intersection outside a locus of codimension at least~3. \end{rema}

\subsection{Embeddings of moduli spaces}\label{subsec:EmbeddingsModuliSpaces}

Let $(S,h)$ be a polarized K3 surface of genus $2$ such that $\NS(S)=\ZZ h$.
We denote by $f\colon S \to \P^2$ the associated double cover, ramified on a very general sextic curve $\mathsf{\Gamma}\subset \PP^2$, and by $\tau_S$ the covering involution on $S$.

Let $n\equiv 0 \pmod 4$ and consider the involutions $\tau_M$, on the moduli space $M_{S,n}$, and $\tau_{\overline{M}}$, on $\overline{M}_{S,n}$, induced by 
$\tau_S$.
The two involutions $\tau_M$ and $\tau_{\overline{M}}$ satisfy $\tau_{\overline{M}}\circ\phi=\phi\circ\tau_M$, where $\phi\colon M_{S,n}\to \overline{M}_{S,n}$ is the divisorial contraction of Section \ref{subsec:DUYK3}, and $\tau_M$ preserves the stratification of the exceptional divisor $\delta$ given in Lemma~\ref{lem:DUYdivisor}\ref{enum:DUYdivisor3}.

The pull-back morphism induces a closed embedding
\[
\iota:= f^*\colon M_{\PP^2,n} \longhookrightarrow M_{S,n}.
\]

\begin{lemm}
There exists a morphism $\overline{\iota}$  such that the following diagram is commutative:
\[
\begin{tikzcd}
M_{\PP^2, n}\arrow[hookrightarrow, r,"\iota"] \arrow[d,"\varphi_{\mathbb{P}^2}"] & M_{S, n} \arrow[d, "\varphi"]\\
\overline{M}_{\PP^2, n}\arrow[r,"\overline{\iota}"]& \overline{M}_{S, n}
\end{tikzcd}
\]
\end{lemm}

\begin{proof}
Let 
\begin{equation*}
R(\lambda_S)\coloneqq \bigoplus_{k=0}^{\infty} H^0(M_{S,n},\cO_{M_{S,n}}(\lambda_S)^{\otimes k}),\qquad
R(\lambda_{\PP^2})\coloneqq \bigoplus_{k=0}^{\infty} H^0(M_{\PP^2,n},\cO_{M_{\PP^2,n}}(\lambda_{\PP^2})^{\otimes k}).
\end{equation*}
Since $\iota^{*}\cO_{M_{S,n}}(\lambda_S)\cong \cO_{M_{\PP^2,n}}(\lambda_{\PP^2})$, we have a homomorphism of graded rings
\[
\iota^*\colon R(\lambda_S) \longrightarrow R(\lambda_{\PP^2}).
\]
By Lemma~\ref{lem:DUYdivisor}, we know that $\overline{M}_{S, n}=\Proj R(\lambda_S)$, and similarly $\overline{M}_{\PP^2, n}=\Proj R(\lambda_{\PP^2})$.
Hence, the morphism $\phi\circ\iota$ corresponds to $\Im(\iota^*)\subset R(\lambda_{\PP^2})$. Since $\phi_{\PP^2}$ corresponds to the whole of $R(\lambda_{\PP^2})$, the result follows. 
\end{proof}

The main result of this section is the following.

\begin{theo}\label{thm:embedding}
The morphism $\overline{\iota}\colon \overline{M}_{\PP^2,n}\rightarrow \overline{M}_{S,n}$ is a closed embedding.
Moreover, $\overline{\iota}(\overline{M}_{\PP^2,n})$ is an irreducible component of the fixed locus of $\tau_{\overline{M}}$ in $\overline{M}_{S,n}$ with its reduced induced scheme structure.
\end{theo}   

In order to prove Theorem~\ref{thm:embedding} we need to go through a few preliminary results about the local picture.
Let
\[
F_p = \cE[1] \oplus \left( k(p)\otimes W \right),
\]
be a polystable element in $\overline{M}_{\PP^2,n}$, where $p\in \PP^2$ is a closed point and $W$ is a vector space of dimension $a>0$.
Let $\overline{\iota}(F_p)$ be the polystable object associated to $f^*F_p$.
We recall the notation $V_{S,\overline{\iota}(F_p)}'$ from~\eqref{eqroma24}, and similarly define
\begin{equation}\label{eq:VPprime}
\begin{split}
V_{\PP^2,F_p}' \coloneqq \left(\Hom_{\PP^2}(\cE,k(p))\otimes W\right) &\oplus \left(\Ext^2_{\PP^2}(k(p),\cE)\otimes W^\vee\right)\\
&\oplus \left(\Ext^1_{\PP^2}(k(p),k(p))\otimes \End(W)\right)    
\end{split}
\end{equation}
The group $G_{S,\overline{\iota}(F_p)}$, defined in~\eqref{eq:group}, acts on $V_{S,\overline{\iota}(F_p)}'$. Similarly, the group $G_{\PP^2,F_p}$, defined in \eqref{eq:groupP^2}, acts on $V_{\PP^2,F_p}'$.

Our goal is to define a morphism of groups $G_{\PP^2,F_p}\to G_{S,\overline{\iota}(F_p)}$ and an equivariant morphism
\[
\widetilde{\iota}_p^{\ \prime}\colon V_{\PP^2,F_p}'\longrightarrow V_{S,\overline{\iota}(F_p)}'
\]
so that the induced map
\[
\widehat{\iota}_p^{\ \prime}\colon V_{\PP^2,F_p}'\sslash G_{\P^2,F_p} \longrightarrow V_{S,\overline{\iota}(F_p)}'\sslash G_{S,\overline{\iota}(F_p)}
\]
on the GIT quotients is a closed embedding.
We divide the construction of $\widetilde{\iota}_p^{\ \prime}$ into two cases according to whether the closed point $p$ belongs to $\mathsf{\Gamma}$ or not.

\subsubsection*{Case 1: $p\not\in\mathsf{\Gamma}$}
In this case, we have that
\[
f^*F_p=f^*\cE [1] \oplus \left( k(q_1)\otimes W \right) \oplus \left( k(q_2)\otimes W \right)
\]
is again polystable, where $q_1,q_2\in S$ are the two distinct points in the preimage $f^{-1}(p)$. Hence $\overline{\iota}(F_p)=f^*F_p$. 
We then have identifications 
\[
V_{S,f^*F_p}' \cong V'_{\PP^2,F_p} \oplus V'_{\PP^2,F_p},
\]
and
\[
G_{\P^2,F_p}\cong \GL(W), \qquad G_{S,f^*F_p} \cong \GL(W) \times \GL(W).
\]
With respect to these identifications, we consider the diagonal morphism 
\begin{equation}\label{eq:DefOfIotaTilde}
\begin{matrix}
\widetilde{\iota}_p^{\ \prime}\colon V'_{\PP^2,F_p} & \longrightarrow & V'_{\PP^2,F_p} \oplus V'_{\PP^2,F_p} \cong V_{S,f^*F_p}' \\
\ x & \longmapsto & \!\!\!\!\!\!\!\!\!\!\!\! (x,x)
\end{matrix}.
\end{equation}
This map is equivariant with respect to the diagonal inclusion
\[
G_{\P^2,F_p}\cong \GL(W) \longhookrightarrow \GL(W) \times \GL(W) \cong G_{S,f^*F_p},     
\]
and hence descends to a morphism on GIT quotients
\[
\widehat{\iota}_p^{\ \prime}  \colon V_{\PP^2,F_p}'\sslash G_{\P^2,F_p} \lra V_{S,\overline{\iota}(F_p)}'\sslash G_{S,\overline{\iota}(F_p)}. 
\]

\begin{lemm}\label{lem:PointIsNotInGamma}
Suppose that $p\notin \mathsf{\Gamma}$. 
Then the morphism $\widehat{\iota}_p^{\ \prime}$ is a closed embedding. 
\end{lemm}

\begin{proof}
By an elementary computation, the induced morphism at the level of invariant rings
\[
\CC[V'_{\PP^2,F_p}\oplus V'_{\PP^2,F_p}]^{\GL(W) \times \GL(W)} \longrightarrow \CC[V'_{\PP^2,F_p}]^{\GL(W)}
\]
is surjective and so $\widehat{\iota}_p^{\ \prime}$ is indeed a closed embedding.
\end{proof}

\begin{rema}\label{rmk:0Case1}
Suppose that $p\notin \mathsf{\Gamma}$. 
Then $\widehat{\iota}_p^{\ \prime}$ maps $[0]$ to $[0]$ because $\widetilde{\iota}_p^{\ \prime}(0)=0$.
\end{rema}

\subsubsection*{Case 2: $p\in\mathsf{\Gamma}$} The case when $p\in\mathsf{\Gamma}$ is more complicated because $f^*F_p$ is no longer polystable, so we need to worry about the polystable reduction when identifying the various vector spaces and group actions. 
In this case, we have
\[
f^*F_p = f^*\cE[1] \oplus \left(f^*k(p) \otimes W\right).
\]
Hence, the polystable reduction is given by
\[
\overline{\iota}(F_p) = f^*\cE[1] \oplus \left(k(q) \otimes W_S\right),
\]
where 
\begin{equation*}\label{eq:WS}
W_S:=W\otimes H^0(S,f^*k(p))    
\end{equation*}
and $q\in S$ is the only point in the preimage $f^{-1}(p)$.
As a consequence, we have
\begin{equation}\label{eq:VSprime}
\begin{split}
V'_{S,\overline{\iota}(F_p)} = \left(\Hom_{S}(f^*\cE,k(q))\otimes W_S\right) &\oplus \left(\Ext^2_{S}(k(q),f^*\cE)\otimes W_S^\vee\right)\\
&\oplus \left(\Ext^1_{S}(k(q),k(q))\otimes \End(W_S)\right).    
\end{split}
\end{equation}

In order to define the morphism $\widetilde{\iota}_p^{\ \prime}\colon V_{\PP^2,F_p}'\to V_{S,\overline{\iota}(F_p)}'$, we first need to look at the action of $\mu_2$ on $V_{S,\overline{\iota}(F_p)}'$ induced by the involution $\tau_{S}$.
We recall from the introduction the notation $(\CC_{+},\rho_+)$ and $(\CC_{-},\rho_-)$ for the two irreducible representations of $\mu_2$.
To start with, the action of $\mu_2$ on $W$ is trivial and on $H^0(S,f^*k(p))$ we have a decomposition
\[
H^0(S,f^*k(p)) = \C_+ \oplus \C_-.
\]
These induce a decomposition
\begin{equation}\label{eq:Decomposition1}
W_S=\left[W\otimes \CC_{+}\right] \oplus \left[W\otimes \CC_{-}\right].
\end{equation}
Moreover, we have 
\begin{equation}\label{eq:ActionInvolutionHomExt}
\begin{split}
&\Hom_{S}(f^*\cE,k(q)) \cong \CC_{+}^{\oplus 2}\\
&\Ext^2_{S}(k(q),f^*\cE) \cong \CC_{-}^{\oplus 2},
\end{split}    
\end{equation}
where $f^*\cE$ has the natural $\mu_2$-linearization and $k(q)$ the trivial one.
In equation~\eqref{eq:ActionInvolutionHomExt}, the first equality is a direct computation and the second equality follows from the first by Serre duality and the fact that the covering involution $\tau_S$ is anti-symplectic. 
This gives the decompositions
\begin{equation}\label{eq:ActionInvolutionHomExtW}
\begin{split}
&\Hom_{S}(f^*\cE,k(q))\otimes W_S\cong \left[W\otimes \CC_{+}^{\oplus 2}\right] \oplus \left[W\otimes \CC_{-}^{\oplus 2}\right]\\
&\Ext^2_{S}(k(q),f^*\cE)\otimes W_S^\vee\cong \left[W^\vee\otimes \CC_{+}^{\oplus 2}\right] \oplus \left[W^\vee\otimes \CC_{-}^{\oplus 2}\right].    
\end{split}    
\end{equation}
The key part to analyze is the last summand of $V_{S,\overline{\iota}(F_p)}'$ in \eqref{eq:VSprime}.
Let  $(x,y)$ be local coordinates near $q\in S$  such that the generator of $\mu_2$ acts as follows:
\[
x\longmapsto x\quad\text{ and }\quad y\longmapsto -y.
\]
Then we have the identification
\begin{equation*}\label{eq:ExtDecomposition}
\Ext^1_{S}(k(q),k(q)) = \CC\cdot x \oplus \CC\cdot y = \CC_{+} \oplus \CC_{-}.    
\end{equation*}
It follows that the invariant part of $\Ext^1_{S}(k(q),k(q))\otimes \End(W_S)$, with respect to the induced decomposition of $\End(W_S)$ coming from~\eqref{eq:Decomposition1}, is
\[
\left[x \otimes \begin{pmatrix}\End(W)&0\\0&\End(W)\end{pmatrix}\right] \oplus \left[y\otimes \begin{pmatrix}0&\End(W)\\\End(W)&0\end{pmatrix}\right].
\]

We now define the morphism $\widetilde{\iota}_p^{\ \prime}\colon V_{\PP^2,F_p}'\to V_{S,\overline{\iota}(F_p)}'$, using the decomposition~\eqref{eq:VPprime}, as follows.
The adjoint morphism gives the identification
\[
\Hom_{\PP^2}(\cE,k(p))=\Hom_{S}(f^*\cE,k(q))
\]
and, by~\eqref{eq:ActionInvolutionHomExt}, $\mu_2$ acts trivially on the right hand side.
By taking duals, we get the identification
\[
\Ext^2_{\PP^2}(k(p),\cE)=\Ext^2_{S}(k(q),f^*\cE)
\]
and, by~\eqref{eq:ActionInvolutionHomExt}, the generator of $\mu_2$ acts as multiplication by $(-1)$ on the right hand side. 
The inclusions of the invariant summand $W=W\otimes\C_+\hookrightarrow W_S$ (see~\eqref{eq:Decomposition1}) and $W^{\vee}=W^{\vee}\otimes\C_-={\rm Ann}(W\otimes\CC_{+})\hookrightarrow W_S^{\vee}$ give  embeddings
\begin{equation}\label{eq:DefOfIotaHatFirstTwoFactors}
\begin{split}
&\alpha\colon \Hom_{\PP^2}(\cE,k(p))\otimes W \longhookrightarrow \Hom_{S}(f^*\cE,k(q))\otimes W_S,\\    
&\beta\colon \Ext^2_{\PP^2}(k(p),\cE)\otimes W^\vee\longhookrightarrow \Ext^2_{S}(k(q),f^*\cE)\otimes W_S^\vee,
\end{split}    
\end{equation}
with images the $(+1)$-eigenspaces of the involution in the decompositions~\eqref{eq:ActionInvolutionHomExtW}.

Lastly, let  $(w,z)$ local coordinates on $\PP^2$ near $p$ such that the morphism $S\to\PP^2$ is  given locally by
\[
(x,y)\longmapsto (w,z)=(x,y^2).
\]
We have the identification
\[
\Ext^1_{\PP^2}(k(p),k(p))\otimes \End(W)=\left[w\otimes \End(W)\right] \oplus \left[z\otimes \End(W)\right].
\]
We define the morphism
\begin{multline}
\left[w\otimes \End(W)\right] \oplus \left[z\otimes \End(W)\right] \xrightarrow{{\gamma}} \\
\xrightarrow{\gamma}\left[x \otimes \begin{pmatrix}\End(W)&0\\0&\End(W)\end{pmatrix}\right] \oplus \left[y\otimes \begin{pmatrix}0&\End(W)\\ \End(W)&0\end{pmatrix}\right]
\end{multline}
%
by setting
\begin{equation}\label{eq:MorphismPhi}
\gamma\left(w\otimes A,z\otimes B\right)\coloneqq\left(x \otimes \begin{pmatrix}A&0\\0&A\end{pmatrix},y\otimes \begin{pmatrix}0&\mathrm{Id}_W\\ B&0\end{pmatrix}\right).
\end{equation}

\begin{defi}
Let $\widetilde{\iota}_p^{\ \prime}\colon V_{\PP^2,F_p}'\to V_{S,\overline{\iota}(F_p)}'$ be the morphism defined by
\begin{equation}\label{eq:DefOfIotaTilde2}
\widetilde{\iota}_p^{\ \prime}
\left(v,\phi,w\otimes A,z\otimes B\right) \coloneqq \left(\alpha(v),\beta(\phi), \gamma(w\otimes A,z\otimes B)\right).  %
\end{equation}
where $v\in\Hom_{\PP^2}(\cE,k(p)\otimes W)$, $\phi\in\Ext^2_{\PP^2}(k(p)\otimes W,\cE)$,  $A,B\in\End(W)$, and we are referring 
to the decomposition in~\eqref{eq:VPprime}.
\end{defi}

Notice that $\widetilde{\iota}_p^{\ \prime}$ is affine and not linear.
The motivation underlying the above definition is the need to define globally a square root of the coordinate $z$, see the proof of Proposition~\ref{prop:ennebarnorm}.

The morphism $\widetilde{\iota}_p^{\ \prime}$ is equivariant with respect to the group homomorphism
\[
\GL(W) \longrightarrow \GL(W_S) \qquad M\mapsto M\otimes\mathrm{Id}_{H^0(S,f^*k(p))}.
\]
It follows that it descends to a morphism of GIT quotients
\[
\widehat{\iota}_p^{\ \prime}\colon V_{\PP^2,F_p}'\sslash G_{\P^2,F_p} \lra V_{S,\overline{\iota}(F_p)}'\sslash G_{S,\overline{\iota}(F_p)}. 
\]

\begin{lemm}\label{lem:PointIsInGamma}
Suppose that $p\in\mathsf{\Gamma}$. 
Then the morphism $\widehat{\iota}_p^{\ \prime}$ is a closed embedding.
\end{lemm}

\begin{proof}
To check that $\widehat{\iota}_p^{\ \prime}$ is a closed embedding, as before, we need to study the induced morphism at the level of invariant rings
\begin{equation}\label{eq:InducedMorphismInvariants}
\CC[V_{S,\overline{\iota}(F_p)}']^{\GL(W_S)}\longrightarrow \CC[V_{\PP^2,F_p}']^{\GL(W)}.
\end{equation}
We use \cite[Theorem 12.1]{procesi:invariants}: the ring of invariants for the action of $\mathrm{SL}(W)$ on $V'$ (this is denoted by $T_{2,2,2}$ in~\emph{loc.~cit.}) is generated by the following functions.
Let us denote a vector in $V_{\PP^2,F_p}'$ by $(v_1,v_2,\phi_1,\phi_2,M_1,M_2)$; then the generators of the ring of invariants are of the form:
\begin{enumerate}[{\rm (a)}]
    \item\label{procesi1} $\mathrm{Tr}(A)$, 
    \item\label{procesi2} scalar products $\phi_jAv_i$,
    \item\label{procesi3} brackets $[A^{(1)}v_{i_1},\dots,A^{(n)} v_{i_n}]$,
    \item\label{procesi4} brackets $[\phi_{i_1}B^{(1)},\dots,\phi_{i_n} B^{(n)}]$,
\end{enumerate}
where $A,A^{(1)},\dots,A^{(n)},B^{(1)},\dots,B^{(n)}$ are non-commutative monomial in the matrices $M_1$ and $M_2$, $v_i,v_{i_1},\dots,v_{i_n}$ are vectors, $\phi_j,\phi_{i_1},\dots,\phi_{i_n}$ are covectors, with $i,j,i_1,\dots,i_n\in\{1,2\}$.

The generators of the form~\ref{procesi3} and~\ref{procesi4} are not invariant under $\GL(W)$: they are only relative invariants.
To get an invariant we need to multiply them.
As remarked in the proof of~\emph{loc.~cit.}, the product
\[
[A^{(1)}v_{i_1},\dots,A^{(n)} v_{i_n}] \cdot [\phi_{i_1}B^{(1)},\dots,\phi_{i_n} B^{(n)}]
\]
can be written as a linear combination of products of scalar products $\phi_{i_l}B^{(l)}A^{(m)}v_{i_m}$, and so in our case, for the group $\GL(W)$, we only need to consider invariants of the form~\ref{procesi1} and~\ref{procesi2}.

Let us write
\[
A=M_1^{a_1}M_2^{b_1}\dots M_1^{a_r}M_2^{b_r}
\]
and let us denote a vector in $V_{S,\overline{\iota}(F_p)}'$ by
\[
\left(\widehat{v_1},\widehat{v_2},\widehat{\phi_1},\widehat{\phi_2},\widehat{M_1},\widehat{M_2}\right).
\]
We consider the invariant functions in $\CC[V_{S,\overline{\iota}(F_p)}']^{\GL(W_S)}$ given by
\begin{equation}\label{eq:LiftingInvariants}
\frac 12 \mathrm{Tr} (\widehat{A})\quad \text{ and }\quad \widehat{\phi}_j \widehat{A} \widehat{v}_i,
\end{equation}
where
\[
\widehat{A} := \widehat{M}_1^{a_1}\widehat{M}_2^{2b_1}\dots \widehat{M}_1^{a_r}\widehat{M}_2^{2b_r}.
\]
Composing the two invariant functions in~\eqref{eq:LiftingInvariants} with $\widehat{\iota}_p^{\ \prime}$, we obtain the invariants of type~\ref{procesi1} and~\ref{procesi2} on $V_{\PP^2,F_p}'$.
This shows that the morphism~\eqref{eq:InducedMorphismInvariants} is surjective, namely $\widehat{\iota}_p^{\ \prime}$ is a closed embedding, as we wanted.
\end{proof}

\begin{rema}\label{rmk:0Case2}
Suppose that $p\in\Gamma$. 
Then 
$\widehat{\iota}_p^{\ \prime}$ maps $[0]$ to $[0]$, although
\[
\widetilde{\iota}_p^{\ \prime}(0)=\left(0,0,0,\begin{pmatrix}0&\mathrm{Id}_W\\ 0&0\end{pmatrix}\right).
\]
Indeed, let $\lambda$ be the $1$-parameter subgroup of $\GL(W_S)$ which acts as the identity on $W\otimes\CC_{+}$ and as multiplication by $t$ on  $W\otimes\CC_{-}$. 
Then the closure of $\{\lambda(t)\widetilde{\iota}_p^{\ \prime}(0) \,:\, t\in \CC^{*}\}$ contains $0$.
\end{rema}

Having studied the local picture, the idea of the proof of Theorem~\ref{thm:embedding} is to show that the image of $\overline{\iota}$ is normal. Since $\overline{\iota}$ is bijective onto its image, this will allow us to conclude.
We start by showing the second part of the statement in Theorem~\ref{thm:embedding}.

\begin{prop}\label{prop:irrcmpfix}
The closed subset $\overline{N}_{S,n}\coloneqq\overline{\iota}(\overline{M}_{\PP^2,n})\subset\ov{M}_{S,n}$ is an irreducible component of the fixed locus of $\tau_{\overline{M}}$.
It is the unique irreducible component containing $[f^*\cE]$, for $[\cE]\in\overline{M}_{\PP^2,n}$ a stable vector bundle.
\end{prop}

\begin{proof}
Since $\tau_{\overline{M}}$ is anti-symplectic, the intersection of its fixed locus with the regular locus of $\overline{M}_{S,n}$ is of pure dimension $n$.
The image of $\overline{\iota}$, which is fixed by $\tau_{\overline{M}}$, has pure dimension $n$ and intersects the regular locus of $\overline{M}_{S,n}$ in those stable vector bundles on $S$ which are pull-back of vector bundles from $\PP^2$. 
The proposition follows. 
\end{proof}

We need to expand slightly the characterization of $\ov{N}_{S,n}$ in Proposition~\ref{prop:irrcmpfix}:

\begin{lemm}\label{lem:UniqueFixedComponent}
Let $[F]\in\overline{M}_{\PP^2,n}$ be such that
\[
F=\cE[1]\oplus \bigoplus_{j=1}^m\, k(p_j),
\]
where $p_1,\dots,p_m\in \PP^2\setminus\mathsf{\Gamma}$
are distinct. Then  $\overline{N}_{S,n}$ is the unique irreducible component of the fixed locus of $\tau_{\overline{M}}$ which contains $[f^*F]$. 
\end{lemm}

\begin{proof}
Since $\overline{N}_{S,n}$ is an irreducible component of the fixed locus of $\tau_{\overline{M}}$, by Proposition~\ref{prop:irrcmpfix}, and $[f^*F]\in \overline{N}_{S,n}$, it suffices to prove that the fixed locus of $\tau_{\ov{M}}$ is locally irreducible at the point $[f^*F]$.
We can write $f^*F$ as follows:
\[
f^*F=f^*\cE[1]\oplus \bigoplus_{j=1}^m\, \left(k(q_{j,1})\oplus k(q_{j,2})\right)
\]
where $f^{-1}(p_j)=\{q_{j,1},q_{j,2}\}$.
Then, Remark~\ref{rmk:KuranishiDelta1} gives us that the analytic germ of  $\ov{M}_{S,n}$ at $[f^*F]$ is isomorphic to the product 
\[
\Ext^1_S(f^*\cE,f^*\cE)\times\bigtimes_{j=1}^m \Big(\left(Q_{j,1}\times \Ext^1_S(k(q_{j,1}),k(q_{j,1}))\right)\times\left(Q_{j,2}\times \Ext^1_S(k(q_{j,2}),k(q_{j,2}))\right)\Big),
\]
at the origin $0$, where $Q_{j,s}$ are quadric cones of dimension~2, for $j=1,\dots,m$, $s=1,2$.
By definition, the involution $\tau_{\ov{M}}$ is induced by the involution $\tau_S$; it acts on $\Ext^1_S(f^*\cE,f^*\cE)$ in the usual way, while it switches $Q_{j,1}\times \Ext^1_S(k(q_{j,1}),k(q_{j,1})$ and $Q_{j,2}\times \Ext^1_S(k(q_{j,2}),k(q_{j,2})$.
It follows that the analytic germ of the fixed locus of $\tau_{\ov{M}}$ at $[f^*F]$ is irreducible, as we wanted.
\end{proof}

The key result is the following.

\begin{prop}\label{prop:ennebarnorm}
$\overline{N}_{S,n}$ is normal.
\end{prop}

\begin{proof}
Let $[F]\in\overline{M}_{\PP^2,n}$ be such that
\[
F=\cE[1]\oplus \bigoplus_{j=1}^m\, \left( k(p_j)\otimes W_j\right)
\]
where $p_1,\dots,p_m\in \PP^2$ are closed points, and $W_j$ are vector spaces of dimension $a_j>0$. 
Since algebraic normality is equivalent to analytic normality~\cite[Satz 4]{Kuh}, it suffices to prove that the analytic germ of $\overline{N}_{S,n}$ at $\ov{\iota}([F])$ is normal.
Let 
\[
F_{p_j}:=\cE[1]\oplus \left( k(p_j)\otimes W_j\right),\quad j=1,\dots,m
\]
and let us define 
\begin{equation}\label{eq:01122025}
\widetilde{\iota}\colon V_{\PP^2,F}= \Ext^1_{\PP^2}(\cE,\cE)\oplus \bigoplus_{j=1}^m V_{\PP^2,F_{p_j}}' \longhookrightarrow V_{S,\overline{\iota}(F)}=\Ext^1_{S}(f^*\cE,f^*\cE)\oplus \bigoplus_{j=1}^m V_{S,\overline{\iota}(F_{p_j})}'    
\end{equation}
by setting
\[
\begin{split}
\widetilde{\iota}&\left(\eta,v_1,\phi_1,w\otimes A_1,z\otimes B_1,\dots,v_m,\phi_m,w\otimes A_m,z\otimes B_m\right)\coloneqq \\
&\left(f^*\eta,\widetilde{\iota}_{p_1}^{\ \prime}(v_1,\phi_1,w\otimes A_1,z\otimes B_1),\dots,\widetilde{\iota}_{p_m}^{\ \prime}(v_m,\phi_m,w\otimes A_m,z\otimes B_m)\right),    
\end{split}
\]
where $\eta\in\Ext^1_{\PP^2}(\cE,\cE)$ and, for $j=1,\dots,m$, $(v_j,\phi_j,w\otimes A_j,z\otimes B_j)\in V_{\PP^2,F_{p_j}}'$ and $\widetilde{\iota}_{p_j}^{\ \prime}$ was defined in~\eqref{eq:DefOfIotaTilde} and~\eqref{eq:DefOfIotaTilde2}, according to whether $p_j$ is in $\mathsf{\Gamma}$ or not.

Since we have the isomorphisms of groups
\[
G_{\mathbb{P}^2,F}\cong\prod_{j=1}^m\mathrm{GL}(W_j),
\]
preserving the decomposition in~\eqref{eq:01122025}, we have an isomorphism at level of GIT quotients
\[
\left(V_{\PP^2,F}\sslash G_{\P^2,F},[0]\right)\cong\left(\Ext^1_{\PP^2}(\cE,\cE),[0]\right)\times\bigtimes_{j=1}^m \left(V_{\PP^2,F_{p_j}}'\sslash G_{\P^2,F_{p_j}},[0]\right).
\]
Similarly, we have an analogous decomposition for $G_{S,\overline{\iota}(F)}$ and $V_{S,\overline{\iota}(F)}\sslash G_{S,\overline{\iota}(F)}$.
Hence, we can apply Lemma~\ref{lem:PointIsNotInGamma} and Lemma~\ref{lem:PointIsInGamma}, and deduce that the morphism $\widetilde{\iota}$ induces a closed embedding
\[
\widehat{\iota}\colon V_{\PP^2,F}\sslash G_{\P^2,F} \longhookrightarrow V_{S,\overline{\iota}(F)}\sslash G_{S,\overline{\iota}(F)},
\]
with the property that $\widehat{\iota}([0])=[0]$, by Remarks~\ref{rmk:0Case1} and~\ref{rmk:0Case2}.
Moreover, since  $\mu_{S,\overline{\iota}(F)}\circ\widehat{\iota}=\mu_{\PP^2,F}$, the closed embedding $\wh{\iota}$  induces a closed embedding
\[
\overline{\iota}'\colon\left(\mu_{\PP^2,F}^{-1}(0)\sslash G_{\P^2,F},[0]\right) \longhookrightarrow \left(\mu_{S,\overline{\iota}(F)}^{-1}(0)\sslash G_{S,\overline{\iota}(F)},[0]\right).
\]
By Proposition~\ref{prop:UhlenbeckVsBridgelandP2} and Lemma~\ref{lem:KuranishiP2}, the domain of $\overline{\iota}'$ is normal, hence it suffices to prove that the image of $\overline{\iota}'$ is the germ of $ \overline{N}_{S,n}$ at $[F]=[0]$. 

By definition, the morphism $\widetilde{\iota}$ has image in the $\mu_2$-invariant part of $V_{S,\overline{\iota}(F)}$, hence the image of $\overline{\iota}'$ is contained in the fixed locus of $\mu_2$. It follows by Lemma~\ref{lem:UniqueFixedComponent} that it suffices to prove that  the image of $\overline{\iota}'$ contains $[f^{*}F']$ where $[F']\in\overline{M}_{\PP^2,n}$ satisfies the hypotheses of the lemma (with $F'$ replacing $F$). 
Suppose first that all the points $p_i$ belong to $\PP^2\setminus\Gamma$. Then 
$\ov{\iota}'([E])=[f^{*}E]$ for all $[E]$ in a neighborhood of $[F]$. Since in an arbitrary small  deformation of $F$ there exists $F'$ satisfying the hypotheses of Lemma~\ref{lem:UniqueFixedComponent}, we are done. Next suppose that there is a single point $p_i$, i.e.~$m=1$, and that $p_1\in\Gamma$. To simplify notation let $p=p_1$. 
We expect that also in this case 
$\ov{\iota}'([E])=[f^{*}E]$ for all $[E]$ in a neighborhood of $[F]$, but since polystabilization is involved we leave aside the problem of proving it. Instead we  examine the restriction of $\wh{\iota}$ to the subspace 
 $\Ext^1_{\PP^2}(k(p),k(p))\otimes \End(W)$ of $V_{\PP^2,F}$.
Let
\[
(w\otimes A,z\otimes B)\in\Ext^1_{\PP^2}(k(p),k(p))\otimes \End(W).
\]
The moment map $\mu_{\PP^2,F}$ restricted to this space is  the commutator
\[
\left[A,B\right]\in\End(W)\cong\Ext^2_{\PP^2}(k(p),k(p))\otimes \End(W).
\]
Hence the restriction of  $\overline{\iota}'$ to $\mu_{\PP^2,F}^{-1}(0)\cap \Ext^1_{\PP^2}(k(p),k(p))\otimes \End(W)$ is identified with the  germ at   $a\cdot [p]$ of the morphism  
\[
\underline{\iota}\colon \mathrm{Sym}^{a}(\PP^2)\longhookrightarrow\mathrm{Sym}^{2a}(S)
\]
induced by the pull-back $f^*$.
Indeed, this last morphism can be described by using matrices as follows (see, e.g.,~\cite[Section 1.4]{Nakajima:book}).
For diagonal matrices $A$ and $B$ of rank~$a$
\[
A=\mathrm{diag}(w_1,\dots,w_a) \qquad B=\mathrm{diag}(z_1,\dots,z_a)
\]
the morphism $\underline{\iota}$ is given by associating the diagonal matrices of rank~$2a$ given by
\begin{equation}\label{eq:Paris15072024}
\mathrm{diag}(w_1,\dots,w_a,w_1,\dots,w_a) \qquad \mathrm{diag}(\sqrt{z_1},\dots,\sqrt{z_a},-\sqrt{z_1},\dots,-\sqrt{z_a}).    
\end{equation}
The matrices on the right hand side of~\eqref{eq:MorphismPhi} are in the same $\GL(W_S)$-orbit as those in~\eqref{eq:Paris15072024}.
We then take a deformation $F'$, corresponding to $A$ and $B$ both diagonal, with distinct non-zero (and small) eigenvalues. It follows that the image of $\ov{\iota}'$ contains $[f^{*}F']$ where $[F']\in\overline{M}_{\PP^2,n}$ satisfies the hypotheses of 
Lemma~\ref{lem:UniqueFixedComponent}, and we are done. 
In general there are some points $p_i$ contained in $\Gamma$ and the remaining  points are contained in $\PP^2\setminus\Gamma$. Since the map  $\ov{\iota}'$ is a product of maps that we have already analyzed, we get that also in this case  there exists an arbitrarily small deformation $F'$ of $F$  satisfying the hypotheses of Lemma~\ref{lem:UniqueFixedComponent} and such that $\ov{\iota}'([F'])=[f^{*}F']$. 
\end{proof}

\begin{proof}[Proof of Theorem~\ref{thm:embedding}.]
The map $\overline{\iota}\colon\overline{M}_{\PP^2,n}\to\overline{N}_{S,n}$ is proper and bijective. 
Since $\overline{M}_{\PP^2,n}$ and $\overline{N}_{S,n}$ are normal  
$\overline{\iota}$ is a closed embedding by Zariski's Main Theorem~\citestacks{03GW}.
As remarked before, the second statement of Theorem~\ref{thm:embedding} is Proposition~\ref{prop:irrcmpfix}.
\end{proof}


\section{Linearization and the fixed locus}\label{sec:linearization}

\subsection{Statement of the main results}\label{subsec:negzerofiber}

In this section we characterize the components of $\Fix(\tau_\lambda)$ in terms of the behavior of a natural lift of the involution $\tau_\lambda$ to the total space of the line bundle $L$ whose first Chern class is $\lambda$. We start in Section~\ref{subsec:LinearizationFamily} with a few general results about linearizations in families. 
In Section~\ref{subsec:LinearizationLagrangian} we then specialize to our case, and show the following.

\begin{prop}\label{prop:ChoiceLinearizationFamily}
Let $(X,\lambda)$ be a polarized {\rm HK} manifold of dimension $2n$ of $\mathrm{K3}^{[n]}$-type, with $q_X(\lambda)=2$ and $\mathrm{div}(\lambda)=2$. Let $L$ be the ample line bundle such that $c_1(L)=\lambda$.
Then there is a  choice of $\mu_2$-linearization of $L$ such that $H^0(X,L)^{\mu_2} = H^0(X,L)$ which varies nicely in polarized families (as described in Lemma \ref{lem:LinearizationFamily}).
\end{prop}

In order to state the second main result of this section, we review a result from our previous paper~\cite{involutions1} on deformations of HK manifolds with involutions.
We keep the notation and assumptions of Sections~\ref{subsec:DUYK3}--\ref{subsec:EmbeddingsModuliSpaces}.
Let  $(S,h)$ be a polarized K3 surface of genus~2 such that $\NS(S)=\ZZ h$.
For $n\equiv 0 \pmod 4$ let  $M_{S,n}$ and  $\phi\colon M_{S,n}\to \overline{M}_{S,n}$  be as in Section~\ref{subsec:DUYK3}. Let $\lambda_S$ and $\delta$ be the divisor classes on $M_{S,n}$ defined  in~\eqref{lamdel}, and let $\overline{\lambda}_S$ be the divisor class on $\overline{M}_{S,n}$ descended from $\lambda_S$ (see Remark~\ref{rmk:AmpleLineBundleDUYK3}).
Note that  the divisility of $\lambda_S$ is $2$ because  $n\equiv 0 \pmod 4$  (see Remark \ref{rmk:divisibility}). 
Let $\tau_M$, $\tau_{\ov{M}}$ be  the  involutions of $M_{S,n}$ and  $\overline{M}_{S,n}$ respectively which were defined in Section~\ref{subsec:EmbeddingsModuliSpaces}. 

The above $M_{S,n}$, $\lambda_S$, $\delta$, and $\tau_M$ satisfy requirements (a)--(d) of~\cite[Section~2]{involutions1}, hence we can apply \cite[Proposition~2.1 \& Corollary~2.2]{involutions1} (which are based on \cite{Namikawa:Def, markman:modular}) to deduce the following. 
First the locus $\Def(\overline{M}_{S,n},\tau_{\overline{M}})$ parametrizing deformations of the pair $(\overline{M}_{S,n},\tau_{\overline{M}})$ is smooth of codimension~1 in $\Def(\overline{M}_{S,n})$.

Secondly, consider the pull-back $\mathcal{M}\to \mathbb{D}$ of the universal family to a general smooth curve $0\in \mathbb{D}\subset \Def(\overline{M}_{S,n},\tau_{\overline{M}})$.
Let us denote by $\tau_{\mathcal{M}}$ the global involution on $\mathcal{M}$; it preserves the fibers over $\mathbb{D}$.
For $t\in \mathbb{D}\setminus\{0\}$, the pair $(\mathcal{M}_t, \tau_t)$ is a smoothing of $(\mathcal{M}_0, \tau_0)=(\overline{M}_{S,n},\tau_{\overline{M}})$ and there is a line bundle $\cL$ on $\mathcal{M}$ with the following properties:
\begin{itemize}
\item $\mathrm{c}_1(\cL_0)=\overline{\lambda}_S$.
\item For $t\in \mathbb{D}$, $\cL_t$ is ample and the eigenspace  $H^2(\tau_t)_{+}\subset H^2(\mathcal{M}_t)$ is $\QQ\cdot  \mathrm{c}_1(\cL_t)$.
\item For $t\in \mathbb{D}\setminus\{0\}$, $q_{\cM_t}(\mathrm{c}_1(\cL_t))=2$ and $\divisore(\mathrm{c}_1(\cL_t))=2$.
\end{itemize}

Now consider the fixed locus $\Fix(\tau_{\cM})\to \mathbb{D}$ of $\tau_{\cM}$ on $\cM$ with the reduced scheme structure.
For $t\ne0$ the fiber $\Fix(\tau_{\mathcal M})_t$ is the disjoint union of two Lagrangian submanifolds of $\cM_t$, by the Main Theorem of \cite{involutions1}.
The fiber at zero satisfies $\Fix(\tau_{\mathcal M})_{0,\mathrm{red}}=\Fix(\tau_{\overline{M}})_{\mathrm{red}}$.
In particular, by Theorem~\ref{thm:embedding}, we have $\overline{\iota}(\overline{M}_{\PP^2,n})\subset \Fix(\tau_{\mathcal M})_0$.
Moreover, the family $\cM\to \mathbb{D}$ is smooth outside the locus
\[
\mathrm{Sing}(\overline{M}_{S,n})\subset \cM_0 \subset \cM
\]
and the map $\Fix(\tau_{\cM})\to \mathbb{D}$ is also smooth outside $\mathrm{Sing}(\overline{M}_{S,n})\cap \Fix(\tau_{\cM})$.
In particular, $\Fix(\tau_{\cM})\to \mathbb{D}$ is smooth at a point in $\overline{\iota}(\overline{M}_{\PP^2,n})$ outside $\overline{\iota}(\mathrm{Sing}(\overline{M}_{\PP^2,n}))\subset \mathrm{Sing}(\overline{M}_{S,n})$.

Let $\cM^{*}\subset\cM$ be the restriction of  $\cM\to\mathbb{D}$ to $\mathbb{D}\setminus\{0\}$. 
Then we have
\[
\Fix_{\cM^{*}}(\tau_{\cM^{*}})= \cY^o_+ \sqcup \cY^o_-,
\]
where
\begin{equation*}
\begin{split}
&\cY^o_+ :=\{ x\in \cM^* \,:\, \mu_2|_{\cL(x)} \cong \CC_{+}\} \\
&\cY^o_- :=\{ x\in \cM^* \,:\, \mu_2|_{\cL(x)} \cong \CC_{-}\}.
\end{split}
\end{equation*}
Here, $\mu_2$ acts on $\cL(x)$ via the $\mu_2$-linearization of $\cL$ of Proposition~\ref{prop:ChoiceLinearizationFamily}. 

We set
\begin{equation*}\label{eq:DefYprimeFamily}
\cY_{\pm} :=\text{closure of $\cY_{\pm}^o$ in $\cM$},\qquad 
\cY:=\cY_{+}\sqcup \cY_{-}.
\end{equation*}

We are now ready to state the second main result of the present section.

\begin{prop}\label{prop:negfiber}
The scheme-theoretic fiber of $\cY_{-}\to\mathbb{D}$ over $0$ is equal to $\ov{\iota}(\ov{M}_{\PP^2,n})$. Moreover, $\cY_{-}$ is normal.
\end{prop}

\subsection{Linearizations in families}\label{subsec:LinearizationFamily}

We formulate a simple result about linearizations of involutions.
In the present section, we work in the analytic category.
Let
\[
\pi\colon \cX \longrightarrow \mathbb{D}
\]
be a flat projective morphism over a smooth pointed analytic space $\mathbb{D}$, with $\cX$ normal and such that $\pi_{*}{\cO}_{\cX}\cong{\cO}_\mathbb{D}$.
We assume there is an involution $\tau_{\cX}$ on $\cX$ preserving the fibers which then induces a $\mu_2$-action on $\cX$.
We start by the following immediate observation:

\begin{lemm}\label{lem:LinearizationGeneral}
Under the above assumptions, suppose  that $\Gamma(\cX, {\cO}_{\cX})=\Gamma(\cX,{\cO}_{\cX})^2$.
For any $\mu_2$-invariant simple sheaf $\cF$ on $\cX$, there is a $\mu_2$-linearization on $\cF$, and  it is  determined up to multiplication by square roots of $1\in  \Gamma(\cX, {\cO}_{\cX})$.
\end{lemm}

In order to apply Lemma~\ref{lem:LinearizationGeneral}, we assume from now on that $\mathbb{D}$ is a disk (in particular, simply connected).
Then we can extract square roots of sections of $\cO_{\cX}$ (as well as in the fibers, by our assumption that $\pi_{*}{\cO}_{\cX}\cong{\cO}_\mathbb{D}$).
We further suppose that there exists a line bundle $\cL$ on $\cX$ which is $\pi$-ample, $\tau_{\cX}$-invariant and such that, for every $t\in\mathbb{D}$ and for all $i>0$, $H^i(\mathcal{X}_t,\mathcal{L}_t)=0$.

\begin{lemm}\label{lem:LinearizationFamily}
Under the above assumptions, we have: 
\begin{enumerate}[{\rm(a)}]
\item\label{enum:LinearizationFamily1} For any $t\in\mathbb{D}$, there is a $\mu_2$-linearization on the line bundle $\mathcal{L}_t$.
\item For any choice of $\mu_2$-linearization on a fiber $\mathcal{L}_{t_0}$ of $\mathcal{L}$, there is a unique $\mu_2$-linearization on $\mathcal{L}$ compatible with the chosen linearization on $\mathcal{L}_{t_0}$.
\item \label{eq:lemmconcl3} For any $\mu_2$-linearization of $\mathcal{L}$, the $\mu_2$-invariant subsheaf $(\pi_*\mathcal{L})^{\mu_2}$ of $\pi_*\mathcal{L}$ is locally free. In particular $\dim H^0(\mathcal{X}_t,\mathcal{L}_t)^{\mu_2}$ is constant for all $t\in\mathbb{D}$.
\end{enumerate}
\end{lemm}

\begin{proof}
The existence of the global $\mu_2$-linearization and on the fibers follows immediately from Lemma~\ref{lem:LinearizationGeneral}.
In particular, for a choice of $\mu_2$-linearization on a given fiber $\mathcal{L}_t$, there is a compatible choice of $\mu_2$-linearization on $\mathcal{L}$.
Lastly, any $\mu_2$-linearization on $\mathcal{L}$ induces a $\mu_2$-linearization on $\pi_*\mathcal{L}$.
Moreover since for all $t\in \mathbb{D}$ we have $H^i(\mathcal{X}_t,\mathcal{L}_t)=0$ for all $i>0$, it follows that $\pi_*\mathcal{L}$ is locally free.
Hence we have $\pi_*\mathcal{L}=(\pi_*\mathcal{L})_+\oplus (\pi_*\mathcal{L})_-$, where $(\pi_*\mathcal{L})_+$ denotes the subsheaf on which $\tau$ acts as $+1$ and $(\pi_*\mathcal{L})_-$ denotes the subsheaf on which $\tau$ acts as $-1$.
It follows that both $(\pi_*\mathcal{L})_+$ and $(\pi_*\mathcal{L})_-$ are locally free.
In particular $(\pi_*\mathcal{L})^{\mu_2}=(\pi_*\mathcal{L})_+$ is locally free. 
\end{proof}

We apply Lemma~\ref{lem:LinearizationFamily} twice: first, to the degeneration studied in~\cite{involutions1} and reviewed in Section \ref{subsec:LinearizationLagrangian} below, and secondly to the family $\cM$ of Section~\ref{subsec:negzerofiber}.
In these examples there exists a $\mu_2$-linearization of $\cL$ such that the $\mu_2$-invariant subsheaf $(\pi_*\mathcal{L})^{\mu_2}$ coincides with $\pi_*\mathcal{L}$.

\subsection{Proof of Proposition~\ref{prop:ChoiceLinearizationFamily}}\label{subsec:LinearizationLagrangian}

We start by reviewing one of the main results from~\cite{involutions1}. As proved in~\cite[Section~5.3]{involutions1}, there exist a family $\cN\to\mathbb{D}$ over a smooth curve $\mathbb{D}$ and a line bundle $\mathcal{L}$ on $\cN$ with the following properties.
For $t\in\mathbb{D}\setminus\{0\}$, the pair $(\cN_t,\mathrm{c}_1(\cL_t))$ is a polarized HK manifold of $\mathrm{K3}^{[n]}$-type with a polarization of square~2 and divisibility~2.
At $t=0$, the fiber $\cN_0$ (denoted by $\overline{M}$) is obtained by a divisorial contraction from a moduli space $M_{\mathrm{last}}$ of Bridgeland stable objects on a very general polarized K3 surface of genus $n$ (see~\cite[Lemma~3.22]{involutions1}).
A natural $\mu_2$-linearization of the fiber $\cL_0$ has been defined in~\cite[Proposition~5.13]{involutions1}.

\begin{prop}\label{prop:H0Invariant}
With respect to the natural $\mu_2$-linearization of $\cL_0$, we have that
\[
H^0(\cN_0,\cL_0)^{\mu_2} = H^0(\cN_0,\cL_0).
\]
\end{prop}

\begin{proof}
We adopt the notation of~\cite{involutions1}. The fiber $\cN_0=\overline{M}$ is given by a divisorial contraction $\phi\colon M_{\mathrm{last}}\to \overline{M}$. 
The $\mu_2$-linearization of $\cL_0$  comes from a natural $\mu_2$-linearization of the line bundle  $L_{\mathrm{last}}$ on $M_{\mathrm{last}}$. The latter 
is induced by a $\mu_2$-linearization of a line bundle $L$ on 
a HK manifold $M$ birational to $M_{\mathrm{last}}$. More precisely, there is a birational map $\psi\colon M\dra M_{\mathrm{last}}$ such that $\psi^{*}L_{\mathrm{last}}\cong L$. Thus 
$H^0(\cN_0,\cL_0)=H^0(M_{\mathrm{last}},L_{\mathrm{last}})=H^0(M,L)$, and hence it suffices to prove that 
$H^0(M,L)^{\mu_2} =H^0(M,L)$. Here the $\mu_2$-action on $M$ is described as follows: $M$ comes with a natural Lagrangian fibration structure $g\colon M\to \P^n$, and the involution maps every Lagrangian fiber to itself and equals multiplication by $-1$ on smooth Lagrangian fibers.
The line bundle $L$ is isomorphic to $\cO_{M}(\Delta)\otimes g^*\cO_{\P^n}(1)$, where $\Delta$ is the divisor described in~\cite[Proposition 3.4]{involutions1}.
The restriction of $\Delta$ to a smooth Lagrangian fiber (which is the Jacobian of a curve of genus~$n$) is twice the natural
principal polarization. 

By~\cite[Theorem~10.32]{Kollar:SingularitiesMMP}, the pushforward $g_* L=g_*\cO_M(\Delta)\otimes\cO_{\P^n}(1)$ is a vector bundle and $(g_* L)^{\mu_2}$ is a vector subbundle, which generically coincides with $g_* L$, and thus they coincide everywhere.
This concludes the proof because every section of twice a
principal polarization is invariant under multiplication by $-1$.
\end{proof}

\begin{proof}[Proof of Proposition~\ref{prop:ChoiceLinearizationFamily}]
Let us consider the family $\mathcal{N}\to\mathbb{D}$.
Lemma~~\ref{lem:LinearizationFamily} shows that analytically locally we can propagate the $\mu_2$-linearization.
By Proposition~\ref{prop:H0Invariant}, the invariant subspace is the whole space at the fiber $\mathcal{N}_0$; hence there is no monodromy and the invariant subspace is the whole space on all fibers, as we wanted.
\end{proof}

\begin{coro}\label{coro:negbase}
Let $(X,\lambda)$ be a polarized {\rm HK} manifold of dimension $2n$ of $\mathrm{K3}^{[n]}$-type, with $q_X(\lambda)=2$ and $\mathrm{div}(\lambda)=2$. 
Let $L$ be the ample line bundle such that $c_1(L)=\lambda$, equipped with the $\mu_2$-linearization of Proposition ~\ref{prop:ChoiceLinearizationFamily}. 
Then $\Fix(\tau_{\lambda})_{-}$ is contained in the base locus of $|L|$.
\end{coro}

\begin{proof}
Let $x\in\Fix(\tau_{\lambda})_{-}$. If $s\in H^0(X,L)$ is invariant under the $\mu_2$-linearization of $L$, then $s(x)=0$. 
Since by construction $H^0(X,L)=H^0(X,L)^{\mu_2}$, the result follows.
\end{proof}

Let  $\cM\to\mathbb{D}$ be the family of Section~\ref{subsec:negzerofiber}, with fiber over $0$ given by $\overline{M}_{S,n}$. Let $\ov{L}$ be the ample line bundle over $\overline{M}_{S,n}$ such that $c_1(\ov{L})=\ov{\lambda}$.
Lemma~\ref{lem:LinearizationFamily} and Proposition~\ref{prop:ChoiceLinearizationFamily} give a natural $\mu_2$-linearization of $\ov{L}$ such that
\[
H^0(\overline{M}_{S,n},\ov{L})^{\mu_2} = H^0(\overline{M}_{S,n},\ov{L}).
\]

\subsection{Proof of Proposition~\ref{prop:negfiber}}\label{subsec:LinearizationDegree2}

We are in the setting of Section~\ref{subsec:negzerofiber}.
Let $f\colon S\to \P^2=|h|^{\vee}$ be the associated double cover, ramified on a smooth sextic curve $\mathsf{\Gamma}$, and let $\tau_S$ be the covering involution of $f$.
Let $\tau_M,\tau_{\ov{M}}$ be the involutions of the moduli space $M_{S,n}$ and of the Donaldson-Uhlenbeck-Yau compactification $\ov{M}_{S,n}$ respectively. Let $\ov{\iota}\colon \ov{M}_{\PP^2,n}\hra \ov{M}_{S,n}$ be the closed embedding of Theorem~\ref{thm:embedding}.

Let
\begin{equation}\label{eq:FixedFamilyFlat}
\pi_{\pm}\colon \cY_{\pm} \longrightarrow \mathbb{D}
\end{equation}
be the restriction of $\cM\to\mathbb{D}$ to $\cY_{\pm}$. 
By the discussion in Section~\ref{subsec:negzerofiber}, $\pi_{\pm}$ is a flat projective morphism which is smooth outside of $\mathrm{Sing}(\overline{M}_{S,n})\cap \cY$.
Moreover the general fiber of  $\pi_{\pm}$ is irreducible, and hence  by~\citestacks{0AY8} we have
\begin{equation}\label{pushstruct}
\left(\pi_{\pm}\right)_*\cO_{\cY_{\pm}}=\cO_{\mathbb{D}}.
\end{equation}

Note that as a consequence of Lemma \ref{lem:LinearizationFamily} and of their definition, $\cY_-$ and $\cY_+$ are disjoint.

\begin{lemm}\label{lem:tuttasola}
The image $\ov{\iota}(\ov{M}_{\PP^2,n})$ in $\ov{M}_{S,n}$ is either the (set theoretic) fiber of $\cY_+\to \mathbb{D}$ over $0$ or the (set theoretic) fiber of $\cY_-\to \mathbb{D}$ over $0$.
\end{lemm}  

\begin{proof}
As already noted, the Zariski dense open subset 
$\ov{\iota}(\ov{M}_{\PP^2,n})\setminus\Sing\overline{M}_{S,n}$ is an irreducible component of the fixed locus of the restriction of $\tau_{\ov{M}}$ to the smooth locus of $\overline{M}_{S,n}$ and hence it is contained in $\cY$. 
Therefore $\ov{\iota}(\ov{M}_{\PP^2,n})$ is an irreducible component of the (set theoretic) fiber of $\cY$ over $0$. 
Of course $\ov{\iota}(\ov{M}_{\PP^2,n})$ belongs either to the fiber over $0$ of  $\cY_{+}$ or of $\cY_{-}$. 
If the former holds we let $\cY_{\sigma}=\cY_{+}$, if the latter holds we let $\cY_{\sigma}=\cY_{-}$.

 It follows from ~\eqref{pushstruct} that the (set theoretic) fiber of $\cY_{\sigma}$ over $0$ is connected. Since $\cY_{+}$ and $\cY_{-}$ are disjoint, it remains to show that $\ov{\iota}(\ov{M}_{\PP^2,n})$ does not meet any other irreducible component of the fiber of $\cY_{\sigma}$ over $0$.

Using the equality in~\eqref{pushstruct}, Grothendieck's Connectedness Theorem ~\cite[Theorem 7]{Faltings:Formal} (see also~\citestacks{0ECQ}) applies to the fibers of $\cY_{\sigma}$:  if $X, X'$ are distinct irreducible components which meet, then the intersection $X\cap X'$ has codimension  $1$ in $X$ (and $X'$).
Now assume that the (set theoretic) fiber of $\cY_{\sigma}$ over $0$ is not equal to $\ov{\iota}(\ov{M}_{\PP^2,n})$. 
Then $\ov{\iota}(\ov{M}_{\PP^2,n})$ meets another irreducible component of the fixed locus of $\tau_{\ov{M}}$, say $X'$, in a subset of codimension $1$ in  $\ov{\iota}(\ov{M}_{\PP^2,n})$. This subset $\ov{\iota}(\ov{M}_{\PP^2,n})\cap X'$ necessarily lies in the singular locus of $\ov{M}_{S,n}$. The intersection $\ov{\iota}(\ov{M}_{\PP^2,n})\cap\Sing\overline{M}_{S,n}$ is contained in the singular locus of $\ov{\iota}(\ov{M}_{\PP^2,n})$ because away from the singular locus of $\ov{M}_{\PP^2,n}$ the map $\ov{\iota}$ equals $\iota$ which is an embedding of  smooth varieties and $\ov{\iota}$ maps $\Sing(\ov{M}_{\PP^2,n})$ into $\Sing\overline{M}_{S,n}$.  However we know that  the singular locus of $\ov{M}_{\PP^2,n}$ has codimension $2$ in $\ov{M}_{\PP^2,n}$, hence since $\ov{\iota}$ is a closed embedding (Theorem \ref{thm:embedding}) it follows that $\ov{\iota}(\ov{M}_{\PP^2,n})$ and $X'$ have non empty intersection away from $\Sing\overline{M}_{S,n}$. 
This is absurd because the fixed locus of $\tau_{\ov{M}}$ is smooth away from $\Sing\overline{M}_{S,n}$.
\end{proof}

Our next goal is to prove that, in the notation of the proof of Lemma~\ref{lem:tuttasola}, we have $\cY_{\sigma}=\cY_{-}$. 
We do not prove this directly, i.e., by computing the action of $\tau_{\ov{M}}$ on the fiber of $\cL_0$  at points of $\ov{\iota}(\ov{M}_{\PP^2,n})$; we follow a more roundabout path. We start by giving an explicit construction of stable sheaves whose moduli points are in the fixed locus of $\tau_M$ but do not belong to $\iota(M_{\PP^2,n})$. 

Let $W\subset S$ be a general subscheme of length $(n+4)/2$. Then 
\begin{equation*}\label{dimext}
\dim\Ext^1_S(\cI_W,\cO_S(-1))=(n-2)/2,
\end{equation*}
and hence 
$\Ext^1_S(\cI_W,\cO_S(-1))$ is non zero because $n\ge 4$. Note that if
\begin{equation}\label{estensionefissa}
0\lra \cO_S(-1)\lra \cF\lra \cI_W\lra 0
\end{equation}
is an extension, then $v(\cF)=v_n$. 
Hence if $\cF$ is stable, then its isomorphism class $[\cF]$ is a point of $M_{S,n}$. 

\begin{lemm}\label{lem:fuoriemmepi}
Keeping notation and hypotheses as above, assume that the extension in~\eqref{estensionefissa} is general. 
Then the following hold:
\begin{enumerate}[{\rm (a)}]
\item\label{enum:fuoriemmepi1} $\cF$ is a stable  locally free sheaf, and  $[\cF]$ is fixed by the involution $\tau_{M}$.
\item\label{enum:fuoriemmepi2} If $C\in|h|$ is a smooth curve  containing exactly one point of $W$ and transverse to $\mathsf{\Gamma}$ at that point, then the restriction of $\cF$ to $C$ is semistable.
\item\label{enum:fuoriemmepi3} $[\cF]$ is not a point of $\iota(M_{\PP^2,n})$.
\end{enumerate}
\end{lemm}

\begin{proof}
Item~\ref{enum:fuoriemmepi1} is easily checked (recall that $\NS(S)=\ZZ h$). 
In order to prove Item~\ref{enum:fuoriemmepi2}, let $P$ be the unique point of $W$ contained in $C$.
The extension in~\eqref{estensionefissa} gives that we have an exact sequence
\begin{equation}\label{restringo}
0\lra \cO_C(P)\otimes\omega^{-1}_C\lra \cF_{|C}\lra \cO_C(-P)\lra 0.
\end{equation}
Since $\deg(\cO_C(P)\otimes\omega^{-1}_C)=\deg\cO_C(-P)$, this proves~\ref{enum:fuoriemmepi2}. 
The exact sequence in~\eqref{restringo} shows that the restriction of $\cF$ to $C$ is not isomorphic to the pull-back of a sheaf on $f(C)$. This proves Item~\ref{enum:fuoriemmepi3}. 
\end{proof}

\begin{prop}\label{prop:tuttasola}
The (set theoretic) fiber of $\pi_{-}\colon\cY_{-}\to\mathbb{D}$ over $0$ is equal to $\ov{\iota}(\ov{M}_{\PP^2,n})$. 
\end{prop}  

\begin{proof}
Assume the contrary. 
By Lemma~\ref{lem:tuttasola} we get that the (set theoretic) fiber of $\pi_{+}\colon\cY_{+}\to\mathbb{D}$ over $0$ is equal to
$\ov{\iota}(\ov{M}_{\PP^2,n})$. Since, by Corollary~\ref{coro:negbase}, $\Fix(\tau_M)_{-}$ is contained in the base locus of $|\cL_0|$, it follows that every $[\cE]\in \Fix(\tau_M)$ which corresponds to a locally free sheaf and is not contained 
in $\iota(M_{\PP^2,n})$ is contained in the base locus of $|\cL_0|$. In particular this must hold for the points $[\cF]$ corresponding to the  locally free sheaves $\cF$ of Lemma~\ref{lem:fuoriemmepi}. This is a contradiction. In fact let $C\in|h|$ be a curve as in Lemma~\ref{lem:fuoriemmepi}\ref{enum:fuoriemmepi2}. 
Since $\cL_0$ is the determinant line bundle on $\ov{M}_{S,n}$ associated to $\cO_S(H)$, one  associates a section $s_{i_{C,*}\xi}$ of $\cL_0$ to a line bundle $\xi$ on $C$  of degree~0 with the property
that $s_{i_{C,*}\xi}(\cE)\neq 0$ if and only if $\Ext^*(\cE,i_{C,*}\xi)=0$.
By Lemma~\ref{lem:fuoriemmepi} we get that for a general line bundle $\xi$ 
we have $\Ext^*(\cF\vert_{C},i_{C,*}\xi)=0$, and hence $[\cF]$ is not contained in the zero divisor of 
$s_{i_{C,*}\xi}$.
\end{proof}

The following result is needed in order to finish the proof of Proposition~\ref{prop:negfiber}.

\begin{lemm}\label{lem:NormalityImpliesReduced}
Let $\pi\colon \mathcal{Z} \to \mathbb{D}$ be a flat projective morphism from an integral variety $\mathcal{Z}$ to a smooth pointed curve $0 \in \mathbb{D}$ and let $\mathcal{Z}_0$ denote the fiber over $0$.
Assume that the non-normal locus of $\mathcal{Z}$ is strictly contained in the fiber $\mathcal{Z}_0$, and that $\mathcal{Z}_0$ is generically reduced and the reduction $\mathcal{Z}_{0,\mathrm{red}}$ is integral and normal.
Then $\mathcal{Z}_0$ is reduced and $\mathcal{Z}$ is normal.
\end{lemm}

\begin{proof}
Let $\nu\colon\widehat{\mathcal{Z}}\to\mathcal{Z}$ be the normalization of $\mathcal{Z}$ and let $\widehat{\pi}\colon \widehat{\mathcal{Z}} \to \mathcal{Z} \to \mathbb{D}$ be the induced morphism.
Since $\widehat{\mathcal{Z}}$ is normal, it satisfies Serre's condition $S_2$, and hence the fiber $\widehat{\mathcal{Z}}_0$ satisfies $S_1$.
Since $\widehat{\mathcal{Z}}_0$ is generically reduced by our hypotheses, it follows that $\widehat{\mathcal{Z}}_0$ is reduced. 
The restriction  $\widehat{\mathcal{Z}}_0 \to \mathcal{Z}_0$ of $\nu$ to the fibers over $0$ factors via a finite proper morphism $\nu_0': \widehat{\mathcal{Z}}_0 \to \mathcal{Z}_{0,\mathrm{red}}$.
This morphism is birational since $\nu$ is an isomorphism over an open set of $\mathcal{Z}$ which intersects $\mathcal{Z}_0$. 
Since $\mathcal{Z}_{0,\mathrm{red}}$ is normal, $\nu_0'$ is an isomorphism, by Zariski's Main Theorem.

We now show that $\mathcal{Z}_0$ is reduced.
Let $h$ be a $\pi$-ample line bundle.
Then $\widehat{h}=\nu^* h$ is $\widehat{\pi}$-ample. 
Since $\mathcal{Z}$ is normal away from $\mathcal{Z}_0$, $\nu$ is an isomorphism away from $\mathcal{Z}_0$, and so
\[
\chi(\mathcal{Z}_t,\mathcal{O}_{\mathcal{Z}_t}(mh_t))=\chi(\widehat{\mathcal{Z}}_t,\mathcal{O}_{\widehat{\mathcal{Z}}_t} (m\widehat{h}_t)),
\]
for $t \neq 0$ and for all $m\in\ZZ$. 
However, since  $\pi$ (resp., $\widehat{\pi}$) is flat, the Hilbert polynomial of  $\mathcal{Z}_t$ (resp., of $\widehat{\mathcal{Z}}_t$) with respect to $h$ (resp., with respect to $\widehat{h}$) is independent of $t$. 
Thus
\[
\chi(\mathcal{Z}_0,\mathcal{O}_{\mathcal{Z}_0}(mh_0))=\chi(\widehat{\mathcal{Z}}_0,\mathcal{O}_{\widehat{\mathcal{Z}}_0} (m\widehat{h}_0))=\chi(\mathcal{Z}_{0,\mathrm{red}},\mathcal{O}_{\mathcal{Z}_{0,\mathrm{red}}}(mh_0)),
\]
for all $m\in\ZZ$.
Then the Hilbert polynomial of the ideal sheaf $\mathcal{I} \subset \mathcal{O}_{\mathcal{Z}_0}$ defining $\mathcal{Z}_{0,\mathrm{red}}$ in $\mathcal{Z}_0$ is trivial and $\mathcal{Z}_0=\mathcal{Z}_{0,\mathrm{red}}$.

Finally, to show that $\mathcal{Z}$ is normal we only need to check Serre's condition $S_2$, since the singular locus of $\mathcal{Z}$ must be in codimension at least~2.
Since $\pi$ is flat, this follows immediately from~\citestacks{00ON} and~\citestacks{0337} applied to the flat local ring homomorphism of local rings $\pi_x^{\sharp}\colon\cO_{\mathbb{D},\pi(x)}\to \cO_{\mathcal{Z},x}$, for all $x\in \mathcal{Z}$.
\end{proof}

\begin{proof}[Proof of Proposition~\ref{prop:negfiber}]
The set theoretic fiber of $\pi_{-}\colon\cY_{-}\to\mathbb{D}$ over $0$ is equal to $\ov{\iota}(\ov{M}_{\PP^2,n})$ by Proposition~\ref{prop:tuttasola}. 
We claim that the hypotheses of Lemma~\ref{lem:NormalityImpliesReduced} hold with $\mathcal{Z}\to\mathbb{D}$ equal to $\pi_{-}\colon\cY_{-}\to\mathbb{D}$. 
In fact $\cY_{-}$ is integral, $\cY_{-}$ is smooth away from the fiber  over $0$, the (scheme theoretic) fiber over $0$ is smooth away from $\Sing \ov{M}_{S,n}$, and the  reduction of the fiber over $0$ (equal to $\ov{\iota}(\ov{M{_{\PP^2,n}}})$) is integral, and normal by Theorem~\ref{thm:embedding}. Hence Proposition~\ref{prop:negfiber} follows from Lemma~\ref{lem:NormalityImpliesReduced}.
\end{proof}


\section{Proof of the negative component result}\label{sec:MainResult}

The goal of this section is to prove Theorem \ref{thm:main1}: for a polarized HK manifold $(X,\lambda)$ of {\rm K3}$^{[n]}$-type, with $q_X(\lambda)=2$ and $\div(\lambda)=2$, the connected component $\Fix(\tau_{\lambda})_-$ of  the fixed locus $\Fix(\tau_\lambda)$ is a Fano manifold of index $3$. 

Let $\pi_{-}\colon\cY_{-}\to \mathbb{D}$ be as in Section~\ref{subsec:negzerofiber}, and let $\cY_{-}(0):=\pi_{-}^{-1}(0)$.
Thus  $\cY_{-}(0)\cong\overline{M}_{\PP^2,n}$ by Proposition~\ref{prop:negfiber}.  
Recall moreover that by Proposition~\ref{prop:negfiber} the variety $\cY_{-}$ is normal. 

\begin{prop}\label{prop:KM}
The normal variety $\cY_-$ is Gorenstein.
\end{prop}

\begin{proof}
As observed in Remark~\ref{rmk:AmpleLineBundleDUYP2}, $\overline{M}_{\PP^2,n}\cong\cY_{-}(0)$ has Gorenstein singularities.
We can then use~\citestacks{0C05}: the morphism $\pi_-$ is flat and has Gorenstein fibers, hence it is Gorenstein.
By~\citestacks{0C11} we then deduce that $\cY_-$ is Gorenstein, as we wanted.
\end{proof}

We consider the two line bundles $\omega_{\cY_-/\mathbb{D}}$ and $\cL |_{\cY_-}$ on $\cY_-$.  
Then $\mathrm{c}_1(\cL |_{\cY_{-}(0)})=\overline{\lambda}_S|_{\overline{M}_{\PP^2,n}}$ is indivisible and (see Remark \ref{rmk:AmpleLineBundleDUYP2})
\[
\omega_{\cY/\mathbb{D}}^\vee |_{\cY_{-}(0)} \cong \left(\cL |_{\cY_{-}(0)}\right)^{\otimes 3}.
\]
Since $\pi_-\colon\cY_-\to \mathbb{D}$ is flat and projective with geometrically integral fibers over the noetherian scheme $\mathbb{D}$, the relative Picard scheme 
\[
\mathrm{Pic}_{\cY_-/\mathbb{D}}^0\lra \mathbb{D}.
\]
 exists (see~\cite[Theorem 4.8]{Kleiman:Picard}).
Moreover, on the special fiber we have
\[
H^1(\cY_{-}(0),\cO_{\cY_{-}(0)})=H^2(\cY_{-}(0),\cO_{\cY_{-}(0)})=0.
\]
By semicontinuity, since all fibers of $\pi_-$ outside $t=0$ are smooth, this equality holds for all fibers.
In particular, by~\cite[Proposition 5.19]{Kleiman:Picard}, the morphism $\mathrm{Pic}_{\cY_-/\mathbb{D}}^0\to \mathbb{D}$ is smooth and projective of relative dimension~0, and hence it is an isomorphism.
This implies both that the line bundle $\cL$ is indivisible on each fiber, since it is so on one fiber, and that the two line bundles $\omega_{\cY_-/\mathbb{D}}^\vee$ and $\cL^{\otimes 3} |_{\cY_-}$ on $\cY_-$ coincide on each fiber, since they coincide on one fiber.

Since  $\omega_X$ and the line bundle $L$ on $X$ such that $c_1(L)=\lambda$  make sense for all $(X,\lambda)$ satisfying the hypotheses of Theorem ~\ref{thm:main1}, one gets the validity of the isomorphism $\omega_{\Fix(\tau_{\lambda})_-}\cong (L^{\vee})^{\otimes 3}\vert_{\Fix(\tau_{\lambda})_-}$  for all such $(X, \lambda)$.


\section{Proof of the positive component result}\label{sec:cubic4folds}

The goal of the present section is to examine the positive component $\Fix(\tau_{\lambda})_+$ of the fixed locus in the case $n=4$, and to prove  Theorem~\ref{thm:main2}.
The proof is similar to that of Theorem~\ref{thm:main1}, with the difference that we need to describe $\pi_{+}\colon \cY_{+}\to\mathbb{D}$ (see~\eqref{eq:FixedFamilyFlat}), in particular the fiber over $0$. 

\subsection{Birational models for the positive components of the fixed locus}\label{subsection:SecondComponentDegree2}

We keep the notation of Section~\ref{subsec:DUYK3}, with $d=1$, $n=4$.
Let $(S,h)$ be a polarized K3 surface such that $\NS(S)=\ZZ h$, $h^2=2$. Denote by $f\colon S\to \P^2$ the associated double cover, ramified on a very general sextic curve $\mathsf{\Gamma}$, and by $\tau_S$ the covering involution on $S$.
We consider the moduli space $M_{S,4}$ together with the natural $\mu_2$-action given by the involution $\tau_M$ induced from $S$.
In the present section we examine the fixed locus of $\tau_M$ via a Mukai flop relating  $M_{S,4}$ and $S^{[4]}$ .

\subsubsection*{A birational model of $M_{S,4}$}

We let $\cO_S(d)$ be the invertible sheaf on $S$ such that $c_1(\cO_S(d))=dh$. 
Let $[W]\in S^{[4]}$. Then
\begin{equation*}
\chi_S(\cI_W,\cO_S(-1))=-\la (1,0,-3),(1,-h,2)\rangle=-1.
\end{equation*}
By Serre duality 
$\Ext^2_S(\cI_W,\cO_S(-1)\cong \Hom_S(\cO_S(-1),\cI_W)^{\vee}$. Since we also have $\Hom_S(\cI_W,\cO_S(-1))=0$, it follows that
\begin{equation}\label{extacca}
\dim\Ext^1_S(\cI_W,\cO_S(-1))=1+\dim H^0(S,\cI_W(1)).
\end{equation}
In particular there exists  a non trivial extension 
\begin{equation}\label{nonspezza}
0\lra\cO_S(-1)\lra\cF\lra \cI_W\to 0. 
\end{equation}

\begin{prop}\label{prop:hirschen}
Let  $\cF$ be a sheaf with $v(\cF)=(2,-h,-1)$. Then $\cF$ is stable if and only if it fits into a non trivial exact sequence  as in~\eqref{nonspezza} (recall that $[W]\in S^{[4]}$).
\end{prop}

\begin{proof}
We start by proving that if $\cF$  fits into a non trivial exact sequence as in~\eqref{nonspezza} then  $\cF$ is stable. Let $\cG\hra\cF$ be a subsehaf of rank $1$. If 
$\cG\hra\cO_S(-1)$ then clearly $\cG\subset\cF$ does not desemistabilize $\cF$. Thus we may assume that the composition 
$\cG\hra\cF\lra \cI_W$ is non zero. If the map $\cG\lra \cI_W$  drops rank in codimension $1$ then, away from codimension $2$, we get an isomorphism $\cG\sim\cO_S(-k)$ for some $k>0$, and hence $\cG\subset\cF$ does not desemistabilize $\cF$. Thus we may assume that the  map $\cG\lra \cI_W$  is an isomorphism  in codimension $1$, and hence  $\cG=\cI_{W_0}\subset \cI_W$, i.e.~$W_0\supset W$. Let $\cQ$ be the cokernel $\cI_W/ \cI_{W_0}$. It follows that we have an exact sequence
\begin{equation*}
0\lra\cO_S(-1)\lra\cF/\cG\lra \cQ\to 0. 
\end{equation*}
Since $\cO_S(-1)$ is locally free and $\cQ$ is Artinian the above exact sequence is trivial, i.e.~$\cQ=Tors(\cF/\cG)$. Let $f\colon\cF\to \cF/\cG$ be the quotient homomorphism, and let $\wt{\cG}\coloneqq f^{-1}(\cQ)$. Then the composition 
$\wt{\cG}\hra\cF\lra \cI_W$ is an isomorphism. This contradicts the hypothesis that the exact sequence in~\eqref{nonspezza} does not split.

Now we prove that if $\cF$ is stable then it fits into a non trivial exact sequence  as in~\eqref{nonspezza}. We have
\begin{equation*}
\chi_S(\cO_S(-1),\cF)=-\la (1,-h,2),(2,-h,-1)\rangle=1.
\end{equation*}
By Serre duality $\Ext^2_S(\cO_S(-1),\cF)\cong \Hom_S(\cF,\cO_S(-1))^{\vee}$, and the latter space vanishes by stability of $\cF$. It follows that 
$\dim\Hom_S(\cO_S(-1),\cF)\ge 1$.
Let  $\alpha\colon\cO_S(-1)\to\cF$ be  a non zero map. By stability of $\cF$, the 
cokernel of $\alpha$ is locally free in codimension $1$ and it has trivial $c_1$. Arguing as above we get that  $\coker(\alpha)$ is torsion free, and hence isomorphic to $\cI_{W}$ for some $0$-dimensional subscheme $W\subset S$. By a Chern class computation, the length of $W$ is $4$. Thus $\cF$
fits into a non trivial exact sequence  as in~\eqref{nonspezza}. 
\end{proof}

Let $W\in S^{[4]}$ be a general point. Then $W$ is not contained in any curve of the linear system $|\cO_S(1)|$. By~\eqref{extacca} there is a unique   sheaf $\cF$ fitting into a non trivial extension as in~\eqref{nonspezza}. 
One defines a birational map
\begin{equation}\label{modbir}
\psi\colon S^{[4]} \dashrightarrow  M_{S,4}
\end{equation}
by mapping  a general $[W]\in S^{[4]}$ to the moduli point of the corresponding sheaf $\cF$. 
Let $Z'\subset S^{[4]}$ be the subset parametrizing subschemes $W$ such that $h^0(S,\cI_W(1))>0$, i.e.~$\dim\Ext^1_S(\cI_W,\cO_S(-1))>1$  (see~\eqref{extacca}. 
Let $Z\subset M_{S,4}$ be the subset parametrizing sheaves $\cF$ such that $h^0(S,\cF(1))>1$.
Note that if $[W]\in Z'$ then $h^0(S,\cI_W(1))=1$, and that if $[\cF]\in Z$ then $h^0(S,\cF(1))=2$. We have dual $\PP^2$-bundles
\begin{equation}\label{fibratopiani}
\xymatrix{ Z'\ar[dr]^{\eta'}   &  &  Z \ar[dl]_{\eta} \\ 
   & M_{S,h}(0,h,-5) & }
\end{equation}
defined as follows. Given $[W]\in Z'$, the map $\eta'$ associates to $[W]$ the isomorphism class of $\cI_{W,C_W}\subset\cO_{C_W}$ in 
$M_{S,h}(0,h,-5)$. Given $[\cF]\in Z$, the map $\eta$ associates to $[\cF]$ the isomorphism class of the cokernel of the map $\cO_S(-1)\otimes H^0(S,\cF(1))^{\vee}\to\cF$.

\begin{prop}\label{prop:ioflop}
The birational map $\psi$ is the flop of the $\PP^2$-bundle $\eta'\colon Z'\to M_{S,h}(0,h,-5)$. The inverse $\psi^{-1}$ is the flop of the $\PP^2$-bundle $\eta\colon Z\to M_{S,h}(0,h,-5)$.
\end{prop}

\begin{proof}
Given the results proved above this is a standard consequence. Alternatively it follows from the general results in~\cite[Section~9]{BM:walls}.
\end{proof}

Let $\pi'\coloneqq\Bl_{Z'}(S^{[4]})\to S^{[4]}$ and $\pi\colon \Bl_Z(M_{S,4})\to M_{S,4}$ be the blow-ups  with centers $Z'$ and $Z$ respectively. By Proposition~\ref{prop:ioflop} the rational map  
$\wt{\psi}\colon \Bl_{Z'}(S^{[4]})\dra\Bl_Z(M_{S,4})$ defined by $\psi$ is an isomorphism, and it fits into a commutative diagram 
\begin{equation}\label{scoppiocontraggo}
\xymatrix{\Bl_{Z'}(S^{[4]}) \ar[d]_{\pi'} \ar[rr]^{\wt{\psi}}  &  &  \Bl_Z(M_{S,4}) \ar[d]_{\pi} \\ 
S^{[4]}\ar@{..>}[rr]^{\psi}   & & M_{S,4}}
\end{equation}
We note that the covering involution $\tau_S$ of $f\colon S\to\PP^2$ induces an involution $\tau_{S^{[4]}}$ of $S^{[4]}$, i.e.~a $\mu_2$ action on $S^{[4]}$.

\begin{prop}\label{prop:azionicompat}
The birational map $\psi$ in~\eqref{modbir} (see also~\eqref{scoppiocontraggo}) is equivariant with respect to the two $\mu_2$-actions. 
\end{prop}

\begin{proof}
Let $[W]\in S^{[4]}$ be a general point. Then  $\psi([W])=[\cF]$, where $\cF$ fits into the unique non split exact sequence in~\eqref{nonspezza}. The involution $\tau_{S^{[4]}}$ maps $[W]$ to $[\tau_S(W)]$ and $\psi$ maps $[\tau_S(W)]$ to the unique sheaf $\cG$ fitting into a non split exact sequence $0\to\cO_S(-1)\to\cG\to \cI_{\tau_S(W)}\to 0$. Since $\tau_S(\cF)$ is such a sheaf, we get that the involution $\tau_M$ of $M_{S,4}$ maps $\psi([W])$ to $\psi([\tau_S(W)])$. This proves the proposition.
\end{proof}

\subsubsection*{Fixed loci of $\tau_{S^{[4]}}$ and of $\tau_{M_{S,4}}$}

Abusing notation, we denote by $\mathsf{\Gamma}$ both the smooth sextic plane curve and its inverse image in $S$.
Let $Y_{\pm}',T'\subset S^{[4]}$ be defined by
\begin{equation*}
Y_+':=\mathsf{\Gamma}^{(4)},\quad Y_-':=f^*(\P^2)^{[2]},\quad
T':=\ov{\{\{x_1,x_2,y,\tau_S(y)\} \,:\, x_1\not= x_2\in
\mathsf{\Gamma},\ y\in(S\setminus\mathsf{\Gamma})\}},
\end{equation*}
where the \lq\lq overline\rq\rq\ means \lq\lq closure\rq\rq.
The decomposition into irreducible components of $\Fix(\tau_{S^{[4]}})$ is as follows:
\begin{equation*}
\Fix(\tau_{S^{[4]}})=Y_+'\sqcup Y_-'\sqcup T'.
\end{equation*}
It is a disjoint union because $\Fix(\tau_{S^{[4]}})$ is smooth. 

By~\cite[Subsection~5.1]{involutions1} there is a bijective correspondence between the sets of irreducible components of $\Fix(\tau_{S^{[4]}})$ and of $\Fix(\tau_{M_{S,4}})$. 
We let $Y_{\pm},T\subset M_{S,4}$ be the irreducible components of $\Fix(\tau_{M_{S,4}})$ corresponding to 
$Y'_{\pm},T'$ respectively. 
Thus the decomposition into irreducible components of $\Fix(\tau_{M_{S,4}})$ is as follows:
\begin{equation*}
\Fix(\tau_{M_{S,4}})=Y_+\sqcup Y_-\sqcup T.
\end{equation*}
We recall the definition of the bijective correspondence. Since $Y'_{+}$ and $T'$ are not contained in the center $Z$ of the flop $\psi$, their strict transforms in $M_{S,4}$ are well-defined;  
$Y_{+},T\subset M_{S,4}$ are the strict transforms $\psi_*Y'_{+},\psi_*T'$ respectively. The irreducible component $Y'_{-}$ is contained in $Z'$. In order to describe $Y_{-}$ we discuss the actions of $\tau_{S^{[4]}}$ and $\tau_{M_{S,4}}$ on $Z'$ and $Z$ respectively. First note that we have an involution
\begin{equation*}\label{azionequattro}
\begin{matrix}
M_{S,h}(0,h,-5) & \overset{\tau_{M(\vv)}}{\lra} & M_{S,h}(0,h,-5) \\
[\cE] & \longmapsto & [\tau_S^{*}\cE]
\end{matrix}
\end{equation*}
where $\vv\coloneqq (0,h,-5)$.

\begin{prop}\label{prop:fissidimdue}
The fixed locus of the involution $\tau_{M(\vv)}$ has two irreducible components, namely
\begin{eqnarray*}
\Sigma & \coloneqq & \{[\iota_{C,*}(\cO_C(-2C))]\,:\, C\in|h|,\ 
\text{$C\overset{\iota_{C}}{\longhookrightarrow} S$ the inclusion}\}, \\
\Omega & \coloneqq & \{[\iota_{C,*}(\cO_C(-C\cdot\Gamma))]\,:\, C\in|h|\}
\end{eqnarray*}
\end{prop}

\begin{proof}
One defines an isomorphism $M(\vv)\xrightarrow{\sim} M_{S,h}(0,h,-1)$ by mapping $[\cE]$ to $[\cE\otimes\cO_S(2h)]$. Under this isomorphism, the involution $\tau_{M(\vv)}$ corresponds to the involution
$\tau$ of~\cite[Proposition~4.1]{involutions1}, and the result follows from that proposition.
\end{proof}

We have $Y'_{-}=(\eta')^{-1}(\Sigma)$, and hence according to~\cite[Proposition~5.1]{involutions1}
\begin{equation}\label{eccoymeno}
Y_{-}=\eta^{-1}(\Sigma)=\{[\cF] \,:\, 0\to \cO_S(-1)^{ 2}\to \cF\to \iota_{C,*}(\cO_C(-2C))\to 0, \text{some $C\in|h|$} \}.
\end{equation}
A word about notation. As a general rule,  letters with superscript, respectively without superscript, denote subsets of $S^{[4]}$, resp. $M_{S,4}$, which correspond via the birational map $\psi$. The symbols for $Y_{\pm}$ have been chosen because they are related to the fibers $\cY_{\pm}(0)$ of $\pi_{\pm}\colon \cY_{\pm}\to\mathbb{D}$ over $0$. 
More precisely $\varphi(Y_{\pm})=\cY_{\pm}(0)$, where $\phi\colon M_{S,4}\to\ov{M}_{S,4}$ is the contraction of Lemma~\ref{lem:DUYdivisor}. 
We will prove this in stages, and the equality $\varphi(Y_{+})=\cY_{+}(0)$ is a key step in the proof of Theorem~\ref{thm:main2}.

\subsubsection*{Geometry of the irreducible components of $\tau_{M_{S,4}}$}
First we analyze  $Y_{+}\subset M_{S,4}$. The intersection $Z'\cap Y_+'$ parametrizes length $4$ subschemes $W\subset\mathsf{\Gamma}$  which lie on a line $R$ (here we view $\mathsf{\Gamma}$ as a plane sextic). 
One checks that the intersection is transverse. Moreover we have an isomorphism
\begin{equation*}
Z'\cap Y_+' \overset{\sim}{\lra} \mathsf{\Gamma}^{(2)}
\end{equation*}
by mapping $[W]$ to the residual of $W$ in 
$R\cap\mathsf{\Gamma}$. 

\begin{prop}\label{prop:ipspiublow}
The restriction of $\psi^{-1}$ to $Y_{+}$  defines a regular map
\begin{equation*}
Y_{+}\overset{\psi^{-1}_{|Y_{+}}}{\lra} Y'_{+}
\end{equation*}
which is the blow-up of $Z'\cap Y_+'$. 
\end{prop}

\begin{proof}
Let $\wh{Y}'_{+}\subset\Bl_{Z'}(S^{[4]})$ be the strict transfrom of $Y'_{+}$. The restriction of $\pi'$ to 
$\wh{Y}'_{+}$ defines a map $\wh{Y}'_{+}\to Y'_{+}$ which is the blow-up of $Z'\cap Y_+'$. The map 
$\wh{Y}'_{+}\to Y_{+}$ given by the restriction of $\pi\circ\wt{\psi}$ is bijective. 
Surjectivity is clear, and also injectivity over $Y_{+}\setminus Z$. Injectivity over $Y_{+}\setminus Z$ follows from~\cite[Proposition~5.2]{involutions1}. Since $Y_{+}$ is smooth, it follows that the restriction of $\pi\circ\wt{\psi}$ to $\wh{Y}'_{+}$ defines an isomorphism $\wh{Y}'_{+}\xrightarrow{\sim} Y_{+}$. 

Alternatively, the result is obtained by analyzing the wall-crossing in the space of stability conditions~\cite[Section~9]{BM:walls}.
\end{proof}

We will need the following result on sheaves parametrized by points of $Y_+$.

\begin{prop}\label{prop:Y+Explicit}
Let $[\cF] \in M_{S,4}$. 
\begin{enumerate}[{\rm (a)}]
\item\label{enum::Y+Explicit1} $[\cF] \in Y_+$ if and only if there exists an exact sequence as in~\eqref{nonspezza} with $W\in\mathsf{\Gamma}^{(4)}$. Moreover such a $W$ is unique.
\item\label{enum::Y+Explicit2} If  $[\cF] \in Y_+$ is general then $\cF$ is locally free.
\item\label{enum::Y+Explicit3} If $[\cF] \in Y_+$ is not locally free then $\cF^{\vee\vee}/\cF$ is of length~1.
\end{enumerate}
\end{prop}

\begin{proof}
We prove~\ref{enum::Y+Explicit1}. 
If $[\cF]\notin Z$ the statement follows at once from Proposition~\ref{prop:hirschen}.
Let us suppose that $[\cF]\in Z$. If there exists an exact sequence as in~\eqref{nonspezza} with $W\in\mathsf{\Gamma}^{(4)}$, then it is clear that $[\cF] \in Y_+$. Now suppose that $[\cF]\in Y_{+}\cap Z$. Since $Y_{-}=\eta^{-1}(\Sigma)$ and $Y_{+}\cap Y_{-}=\es$, it follows that $\eta([\cF])\in \Omega$, where $\Sigma,\Omega$ are as in Proposition~\ref{prop:fissidimdue}. Thus $\eta([\cF])=[\iota_{C,*}(\cO_C(-C\cdot\Gamma))]$. Note that the unicity of $W$ follows from Proposition \ref{prop:ipspiublow}.

We now prove~\ref{enum::Y+Explicit2}. 
Let $[\cF] \in Y_+$ be general. Then it fits into a non trivial exact sequence~\eqref{nonspezza} where $W\in {\mathsf\Gamma}^{(4)}$ is general, and such an extension is unique.  Since $W\in {\mathsf\Gamma}^{(4)}$ is general the extension is locally free. 
(Explicitly: if $W$ is reduced and no $3$ points of $W$ lie on a line, then $\cF$ is locally free.)

In order to prove~\ref{enum::Y+Explicit3} we observe that if $T\coloneqq\cF^{\vee\vee}/\cF$ is of length~$>1$, then it must have length~2 and hence $\cF^{\vee\vee}$ is isomorphic to the unique stable vector bundle with Mukai vector $(2,-h,1)$, i.e.~$f^*\Omega_{\PP^2}(1)$. 
Let us choose a non-zero morphism $\cO_S(-1)\to \cF$.
We have the following diagram:
\begin{equation*}
\xymatrix{ 
&0\ar[d]&0\ar[d]&&\\
&\cO_S(-1)\ar[r]^{=}\ar[d]&f^*\cO_{\PP^2}(-1)\ar[d]&&\\
0\ar[r]&\cF\ar[r]\ar[d]&f^*\Omega_{\PP^2}(1)\ar[d]\ar[r]&T\ar[r]\ar[d]^{=}&0\\
0\ar[r]&\cI_W\ar[r]\ar[d]&f^*\cI_{p}\ar[d]\ar[r]&T\ar[r]&0\\
&0&0&&
}
\end{equation*}
where $p\in\PP^2$ is a closed point.
Since $W\in\mathsf{\Gamma}^{(4)}$, this is impossible since the scheme structure of $f^*\cI_p$ is not in $\mathsf{\Gamma}$.
\end{proof}

Next we describe  $Y_-$ and the sheaves that it parametrizes. 

\begin{prop}\label{prop:ypsilonmeno}
Pull-back defines an isomorphism
\begin{equation*}
\begin{matrix}
M_{\PP^2,4} & \overset{\sim}{\lra} & Y_{-}\\
[\cE] & \longmapsto & [f^{*}\cE]
\end{matrix}
\end{equation*}
\end{prop}

\begin{proof}
The pull-back map $i\colon M_{\PP^2,4}\to M_{S,4}$ is an embedding, and the image is contained in $\Fix(\tau_{M_{S,4}})$.  Since $M_{\PP^2,4}$ is irreducible and projective of dimension $4$, it follows that $i(M_{\PP^2,4})$ is an irreducible component of $\Fix(\tau_{M_{S,4}})$. Thus $i(M_{\PP^2,4})$ equals one of $Y_{\pm},T$.

Let $[\cE]\in M_{\PP^2,4}$. We have an exact sequence
\begin{equation*}
0\lra \cO_{\PP^2}(-1)^{\oplus 2}\lra \cE\lra i_{R,*}(-2)\lra 0,
\end{equation*}
where $i_R\colon  R\hra\PP^2$ is the inclusion of a line. Pulling back to $S$, we get
an exact sequence
\begin{equation*}
0\lra \cO_{S}(-1)^{\oplus 2}\lra f^{*}\cE\lra i_{C,*}(-2\omega_C)\lra 0,
\end{equation*}
where $i_C\colon  C\hra S $ is the inclusion of a divisor $C\in|\cO_S(1)|$. Thus, $[f^{*}\cE]$ belongs to $Y_{-}$ by~\eqref{eccoymeno}.
\end{proof}

Lastly, we deal with the sheaves parametrized by $T$.

\begin{prop}\label{prop:eccoti}
Let $x_1,x_2\in{\mathsf\Gamma}$ be distinct points, and let $\cF$ be a sheaf on $S$ 
fitting into an exact sequence
\begin{equation}\label{eccoeffe}
0\longrightarrow \cF \longrightarrow f^*\Omega_{\P^2}(1) \xlongrightarrow{u} k(x_1)\oplus k(x_2)\longrightarrow 0.
\end{equation}
Then $[\cF]\in T$, and in fact the general point of  $T$ is equal to $[\cF]$ where $\cF$ is as above.
\end{prop}

\begin{proof}
Let $\cF$ be a sheaf as in~\eqref{eccoeffe}. Then $\cF$ is slope stable and $v(\cF)=(2,-h,-1)$, and hence $[\cF]\in M_{S,4}$. Let ${\mathsf\Gamma}^{(2)}_{*}\subset  {\mathsf\Gamma}^{(2)}$ be the complement of the diagonal, and 
let $V\to{\mathsf\Gamma}^{(2)}$ be the $\PP^1\times\PP^1$ bundle with fiber $\PP(f^*\Omega_{\P^2}(1)(x_1))\times(f^*\Omega_{\P^2}(1)(x_1))$ over $(x_1,x_2)$. 
Given $z=(x_1,x_2,l_1,l_2)\in V$ we let $\cF_z$ be the sheaf fitting into an exact sequence  as in~\eqref{eccoeffe} with the kernel of $u$ at $x_i$ equal to $l_i$. We have an injective map
\begin{equation*}
\begin{matrix}
V & \overset{\alpha}{\longhookrightarrow} & M_{S,4} \\
z & \longmapsto & [\cF_z]
\end{matrix}
\end{equation*}
Since $[\cF_z]\in\Fix(\tau_{M_{S,4}})$ and $V$ is irreducible of dimension $4$, the closure $\ov{\alpha(V)}$ is one of $Y_{\pm},T$ (and $\alpha(V)$ is an open dense subset of  $\ov{\alpha(V)}$). Since all points of $\alpha(V)$ parametrize non locally free sheaves, while a general point of $Y_{+}$, or of $Y_{-}$, parametrizes a locally free sheaf, it follows that $\ov{\alpha(V)}=T$. 
\end{proof}

The result below follows at once from Proposition~\ref{prop:eccoti}. 

\begin{coro}\label{coro:eccoti}
Let $[\cF]\in T$. Then $l(\cF^{\vee\vee}/\cF)=2$.
\end{coro}

\subsubsection*{Computations with divisor classes}
We describe $\NS(S^{[4]})$ and $\NS(M_{s,4})$ via Mukai's map.
Let $w:=(1,0,-3)$. We identify $S^{[4]}$ with the moduli space of sheaves on $S$ with Mukai vector $w$ by associating to $[Z]\in S^{[4]}$ the  sheaf $\cI_Z$. 
Let $h_{S^{[4]}},\delta_{S^{[4]}}\in H^2(S^{[4]};\ZZ)$ be the pull-back of the symmetrization $h^{(4)}$ of the class $h$ via the 
Hilbert-Chow morphism $S^{[4]}\to S^{(4)}$ and  the class such that $2\delta_{S^{[4]}}$ is the divisor parametrizing non reduced subschemes respectively. Then
\begin{equation*}
h_{S^{[4]}}:=\theta_w(0,h,0),\quad \delta_{S^{[4]}}:=\theta_w(1,0,3).
\end{equation*}
In fact, given our choice of map $\theta_w$ (see~\ref{mappatheta}) the equalities above follow from a straightforward computation (see~\cite[p.630]{Kieran:weight2}).
Let $\lambda_S,\delta\in \NS(M_{S,4})$ be the elements defined in Equation~\eqref{lamdel}.

\begin{prop}\label{prop:NefConeHilb4}
We have
\begin{equation}\label{nsazione}
\psi_{*}(h_{S^{[4]}}) = 2 \lambda_S - \delta, \qquad \psi_{*}(\delta_{S^{[4]}}) = 3 \lambda_S - 2\delta,
\end{equation}
and $\psi^{-1}$ is associated to the nef divisor $4 \lambda_S - \delta$. 
\end{prop}

\begin{proof}
Let $R\colon H^*(S,\ZZ)\to H^*(S,\ZZ)$ be the reflection in the $(-2)$-vector $v(\cO_S(-1))$. Then $R(w)=v_4$, and hence $R(w^{\bot})=v_4^{\bot}$. Identifying $H^2(S^{[4]},\ZZ)$ and $H^2(M_{S,4},\ZZ)$ with $w^{\bot}$ and $v_4^{\bot}$ respectively via the maps $\theta_w$ and $\theta_v$, the map $\psi_{*}$ is given by the reflection $R$. A straightforward computation gives the equalities in~\eqref{nsazione}.

In order to prove the last statement it suffices to prove that the restriction of $5h_{S^{[4]}}-2\delta_{S^{[4]}}=\psi^{*}(4\lambda_S-\delta)$ to the general fiber of the fibration $\eta'$ in~\eqref{fibratopiani} is trivial. The general fiber is identified with a complete linear system of degree $4$ on a smooth curve $C$ of genus $2$, which we identify with $\PP^2$. The restriction of $h_{S^{[4]}}$ to the plane is $\cO_{\PP^2}(2)$, and the restriction of $2\delta_{S^{[4]}}$ is the subset of the linear system parametrizing non reduced divisors. Since the complete linear system maps $C$ birationally to a quartic curve with a point of order $2$, which has dual curve of degree $10$, we get that the restriction of $2\delta_{S^{[4]}}$ to $\PP^2$ has degree $10$. This finishes the proof.
\end{proof}

\begin{lemm}\label{lem:CanonicalDivisors}
We have:
\[
\omega_{Y_+'} = (3 h_{S^{[4]}} - \delta_{S^{[4]}})\vert_{Y_+'} \quad 
\text{ and }\quad \omega_{Y_+} = (3 \lambda_S - \delta)\vert_{Y_+}.
\]
\end{lemm}

\begin{proof}
Let $\omega^{(4)}$ be the line bundle on $\mathsf{\Gamma}^{[4]}$ obtained by symmetrizing  the canonical line bundle on $\mathsf{\Gamma}^4$ (this is the tensor product of the pull-backs of the canonical bundle of $\mathsf{\Gamma}$ via the four projections). Let $\Diag\subset \mathsf{\Gamma}^{[4]}$ be the divisor parametrizing non reduced subschemes (the \lq\lq diagonal\rq\rq), and let $\Diag/2$ be the ramification divisor class associated to the cover $\mathsf{\Gamma}^4\to \mathsf{\Gamma}^{[4]}$.
We have
\[
\omega_{\mathsf{\Gamma}^{[4]}}=\omega^{(4)}_{\mathsf{\Gamma}}(-\Diag/2).
\]
The first equality of the lemma follows, because  $\omega^{(4)}_{\mathsf{\Gamma}}=3 h \vert_{\mathsf{\Gamma}^{(4)}}$ and $\Diag/2=\delta_{S^{[4]}}\vert_{\mathsf{\Gamma}^{(4)}}$.

Next we compute $\omega_{Y_+}$. Let $\alpha\colon Y_{+}\to Y_{+}'=\mathsf{\Gamma}^{(4)}$ be the restriction of $\psi^{-1}$ to $Y_{+}$. Of course $\alpha$ is identified with the blow up of $\mathsf{\Gamma}^{(2)}$ (see Proposition~\ref{prop:ipspiublow}). 
Let  $E\subset Y_+$ be the exceptional divisor of $\alpha$.
Since $\alpha$ defines an isomorphism between $Y_{+}\setminus E$ and $Y_{+}'\setminus \mathsf{\Gamma}^{(2)}$, there exists an integer $a$ such that
\begin{equation}\label{eq:CanonicalBundleBlowUp}
(3 \lambda_S - \delta)\vert_{Y_+} = (\psi^{-1})^{*}(3 h_{S^{[4]}} - \delta_{S^{[4]}})\vert_{Y_+} = \alpha^{*}(\omega_{Y_{+}'})^*+ a E,  
\end{equation}
by using the first equality of the lemma.
In order to determine $a$, let $R$ be a fiber of the map $E\to \mathsf{\Gamma}^{(2)}$ obtained by restricting $\alpha$. Then $R$ is a line in the fibration $\eta\colon Z\to N$ dual to the fibration $\eta'$ in~\eqref{fibratopiani}. Let $\cl(R)\in H^2(M_{S,4};\QQ)$ be the class corresponding to integration along $R$ via the BBF bilinear symmetric form. We claim that
\begin{equation*}
q_{M_{S,4}}\left(\cl(R),4\lambda_S-\delta\right)=0,\qquad 
q_{M_{S,4}}\left(\cl(R),\lambda_S\right)\ge 0,\qquad 
q_{M_{S,4}}\left(\cl(R),\cl(R)\right)=-\frac{13}{6}.
\end{equation*}
In fact, the first equality holds by the last assertion in
Proposition~\ref{prop:NefConeHilb4}, the inequality in the middle holds because $\lambda_S$ is nef, and the last equality holds by~\cite[Table H4]{HT:IntersectionNumbers}.
Since $\cl(R)$ belongs to the span of $\lambda_S$ and $\delta$, the above equations give that 
\[
\cl(R) = \frac 12 \lambda_S - \frac 23 \delta. 
\]
Hence
\[
\int_R(3\lambda_S - \delta) = -1,
\]
and by~\eqref{eq:CanonicalBundleBlowUp} we get that $a=1$. Thus 
\[
\omega_{Y_+} = \alpha^*(\omega_{Y_+'}) + E = (3 \lambda_S - \delta)\vert_{Y_+}.\qedhere
\]
\end{proof}

\subsubsection*{The irreducible components of the fixed locus in $\overline{M}_{S,4}$}
The next step is to understand how each fixed component behaves with respect to the divisorial contraction $\phi\colon M_{S,4}\to \overline{M}_{S,4}$. We recall that $\Delta_4$ is the exceptional locus of $\phi$, see Section~\ref{subsec:DUYK3}.
We let $\ov{Y}_{\pm}:=\phi(Y_{\pm})$ and $\ov{T}:=\phi(T)$, where $T\subset M_{S,4}$ is the strict transfrom of $T'$. 
Clearly we have
\begin{equation*}
\Fix(\tau_{\ov{M}})=\ov{Y}_{+}\cup \ov{Y}_{-}\cup\ov{T}.
\end{equation*}
Note that the irreducible components may very well meet, because $\ov{M}$ is singular.
By Proposition~\ref{prop:ypsilonmeno} and Proposition~\ref{prop:tuttasola} 
we have
\begin{equation}\label{oppenheimer}
\ov{Y}_{-}=\cY_{-}(0)\cong \ov{M}_{\P^2,4}.
\end{equation}

\begin{prop}\label{prp:bobbysolo}
Keeping notation as above, we have  $\ov{Y}_{+}=\cY_{+}(0)_{\rm red}$. Moreover 
$\ov{Y}_{+}\cap\ov{T}=\es$. 
\end{prop}

\begin{proof}
First we show that  $\ov{Y}_{+}\subset\cY_{+}(0)$.
In fact if $[\cF]\in Y_{+}$ is a general point then $\cF$ is locally free by Proposition \ref{prop:Y+Explicit}. Since the Donaldson--Uhlenbeck--Yau contraction map $\varphi\colon M_{S,n}\rightarrow \overline{M}_{S,n}$ is an isomorphism on the open subset parametrizing locally free sheaves, $[\cF]$ belongs to the closure of ${\rm Fix}_{\cM^*}(\tau_{\cM^*})$ (using the notation of Section \ref{subsec:negzerofiber}). Hence  $\ov{Y}_{+}$ is contained in one of $\cY_{+}(0),\cY_{-}(0)$. By~\eqref{oppenheimer}  we get that $\ov{Y}_{+}\subset\cY_{+}(0)$. Let $[\cF]\in T$: by Corollary~\ref{coro:eccoti} 
we have $l(\cF^{\vee\vee}/\cF=2$. Since for $[\cF]\in Y_{+}$ we have $l(\cF^{\vee}/\cF)\le 1$ (see Proposition~\ref{prop:Y+Explicit}), it follows that 
$\ov{Y}_{+}\cap\ov{T}=\es$. 
Proposition~\ref{coro:eccoti} also gives that $\ov{T}$ is birational to $\mathsf{\Gamma}^{(2)}$, and hence it has dimension $2$. 
Since all irreducible components of $\cY_{+}(0)$ have dimension $4$, it follows that $\ov{Y}_{+}=\cY_{+}(0)_{\rm red}$. 
\end{proof}

Next we  describe locally the scheme $\ov{Y}_+$ at  points $[F]\in\phi(Y_+\cap \Delta_{4,1})$, i.e., such that $F:=\cE[1]\oplus k(q)$. 
By Remark~\ref{rmk:KuranishiDelta1}, locally (in the analytic topology) it is given by a product
\[
\Ext^1_S(\cE,\cE)\times Q \times \Ext^1_S(k(q),k(q)),
\]
where $Q=\Spec(\CC[u_1,u_2,u_3]/(u_1^2-u_2u_3))$ is the quadric cone.
The action of the involution $\overline{\tau}$ is then given as follows.
As $\mu_2$-representations, we have
\begin{align*}
  &\Ext^1_S(\cE,\cE) = \C_{+}^{\oplus 2}\oplus \C_{-}^{\oplus 2}\\
  &\Ext^1_S(k(q),k(q)) = \C_{+}\oplus \C_{-}.  
\end{align*}
For the cone, the action on the variables $u_1,\dots,u_3$ is given by
\begin{equation}\label{azionecono}
u_1\mapsto -u_1,\quad u_2\mapsto u_2,\quad u_3\mapsto u_3.
\end{equation}
To see why we recall that by Remark~\ref{rmk:KuranishiDelta1} we have 
\begin{equation}\label{formulaperu}
u_1=a_1 b_1,\quad u_2=-a_1 b_2,\quad u_3=a_2 b_1,
\end{equation}
where $\{a_1,a_2\}$, $\{b_1,b_2\}$ are dual bases of $\Ext^1(\cE[1],k(p))$ and $\Ext^1(k(p),\cE[1])$ respectively. We claim that the bases $\{a_1,a_2\}$, $\{b_1,b_2\}$  can be chosen so that 
\begin{equation}\label{azioneprecono}
\tau_S^{*}(a_1)=-a_1,\quad \tau_S^{*}(a_2)=a_2,\quad 
\tau_S^{*}(b_1)=b_1,\quad \tau_S^{*}(b_2)=-b_2.
\end{equation}
In fact since $\cE$ is not the pull-back of a vector bundle on $\PP^2$ (note for instance that $c_2(\cE)$ is odd), both the $+1$ and the $-1$ eigenspace for the action of $\tau_S^{*}$ on the fiber $\cE(p)$ has dimensions $1$.
This proves that we can choose a basis $\{a_1,a_2\}$ such that the first two equalities in~\eqref{azioneprecono} hold. Since Serre duality between $\Ext^1(\cE[1],k(p))$ and $\Ext^1(k(p),\cE[1])$ is given by the Yoneda product followed by the trace map, and $\tau_S$ is antisymplectic, it follows that  $\tau_S^{*}$ acts as claimed on the dual basis $\{b_1,b_2\}$.
The action on $u_1,u_2,u_3$ is as in~\eqref{azionecono} by the equalities in~\eqref{formulaperu} and in~\eqref{azioneprecono}.

It follows that the fixed locus is given \emph{as a scheme} by the equation $u_1=0$, namely it is $u_2u_3=0$. 
We have thus proved the following result.

\begin{lemm}\label{lem:AnalyticFixedLocusDim4}
We keep the above notation.
We have an isomorphism of analytic germs
\[
\left(\overline{Y}_+,[\cE[1]\oplus k(q)] \right)\cong \left( \mathbb{A}^3_{\CC} \times B, 0 \right),
\]
where $B=\Spec(\CC[u_2,u_3]/(u_2u_3))$.
\end{lemm}

In particular, by Lemma~\ref{lem:AnalyticFixedLocusDim4}, we have that $\overline{Y}_+$ is a local complete intersection.
The key result is then the following.

\begin{prop}\label{prop:CanonicalBundleFixedLocus}
Let $\widetilde{L}$ be the line bundle on $M_{S,4}$ with associated divisorial contraction $\phi\colon M_{S,4}\to \ov{M}_{S,4}$, see Lemma~\ref{lem:DUYdivisor}. 
We have
\[
(\phi\vert_{Y_+})^*\omega_{\overline{Y}_+} \cong  \widetilde{L}^{\otimes 3}\vert_{Y_+}.
\]
In particular, $\omega_{\overline{Y}_+}$ is ample.
\end{prop}

\begin{proof}
Let us denote by $D$ the intersection $D:=Y_+\cap \Delta_4$.
By Lemma~\ref{lem:AnalyticFixedLocusDim4}, the map $\phi\vert_{Y_+}\colon Y_+\to \overline{Y}_+$ 
defines an isomorphism between $Y_+\setminus D$ and $\phi(Y_+\setminus D)$, and the restriction of $\phi$ to $D$  
is an \'etale 2-1 morphism onto its image.
Hence, we get an isomorphism of line bundles
\[
\omega_{Y_+} \cong (\phi\vert_{Y_+})^*\omega_{\overline{Y}_+} \otimes \cO_{Y_+}(-D).
\]
Since $\cO_{Y_+}(D) \cong \delta\vert_{Y_+}$, the proposition follows from Lemma~\ref{lem:CanonicalDivisors}.
\end{proof}

\subsection{Proof of Theorem \ref{thm:main2}}\label{subsec:ProofMainThm2}

The proof goes along the same lines as that of Theorem~\ref{thm:main1}.

\begin{prop}\label{prop:fixreduced}
We have the equality of schemes $\cY_{+}(0)=\overline{Y}_+$.
\end{prop}

\begin{proof}
The fixed locus $\Fix_{\mathcal{M}/\mathbb{D}}(\tau_{\cM})$ has a natural scheme structure, defined as  the fibered product of $\mathbb{D}$-schemes
\[
\xymatrix{
\Fix_{\mathcal{M}/\mathbb{D}}(\tau_{\mathcal{M}})\ar[d] \ar[rr] &&\mathcal{M} \ar[d]^{\Delta_{\mathcal{X}/B}}\\
\mathcal{M} \ar[rr]^{(\mathrm{id},\tau_{\mathcal{M}})} && \mathcal{M}\times_{\mathbb{D}} \mathcal{M}.
}
\]
With this definition the fixed locus with its scheme structure behaves well with respect to base change; in particular
\[
\Fix_{\mathcal{M}_0}(\tau_{\ov{M}}) = \Fix_{\mathcal{M}/\mathbb{D}}(\tau_{\mathcal{M}}) 
\times_{\mathbb{D}} \Spec(k(0)).
\]
The proposition follows because by~\eqref{oppenheimer} and Proposition~\ref{prp:bobbysolo}, the component  $\overline{Y}_+$ does not meet any other component of the fixed locus of $\tau_{\ov{M}}$.
\end{proof}

By~\citestacks{0C05} and~\citestacks{0C11} it follows that $\cY_+$ is Gorenstein, as in the proof of Proposition~\ref{prop:KM} (and thus normal, since its singular locus coincides with the singular locus of $\overline{Y}_+$).
The theorem now follows from Proposition~\ref{prop:CanonicalBundleFixedLocus} and an argument analogous to that given in the proof of the equality in~\eqref{canmenotre}, see Section~\ref{sec:MainResult}.


\section{The fixed locus if the divisibility is~1}\label{sec:poscomp}

Let $(X,\lambda)$ be a polarized HK manifold of $\mathrm{K3}^{[n]}$-type with $q_X(\lambda)=2$. 
Then there exists a unique involution $\tau_{\lambda}$ of $X$ such that~\eqref{riflessione} holds (regardless of the divisibility of $\lambda$). If $\divisore(\lambda)=1$ then the fixed locus $\Fix(\tau_{\lambda})=\Fix(\tau_{\lambda})_+$ is irreducible, see~\cite[Main Theorem]{involutions1}.
In the present section we motivate the following conjecture.

\begin{conj}\label{conj:GeneralTypeDiv1}
Let $(X,\lambda)$ be a polarized {\rm HK} manifold of $\mathrm{K3}^{[n]}$-type with $q_X(\lambda)=2$
and $\divisore(\lambda)=1$. 
Let $F:=\Fix(\tau_{\lambda})$ and let $L$ be the ample line bundle such that $c_1(L)=\lambda$.
Then in $\Pic(F)$ we have the equality 
\begin{equation}\label{canonicofisso}
\omega_F^{\otimes n}= L|_{F}^{\otimes \frac{(n+2)(n+1)}{2}}.
\end{equation}
\end{conj}

In particular, the conjecture implies that $\Fix(\tau_{\lambda})$ has ample canonical bundle. The conjecture also implies that  the structure sheaf $\cO_{\Fix(\tau_{\lambda})}$ is atomic, see \cite[Theorem 1.8]{beck:Atomic}.
If $n=2$ the equality in~\eqref{canonicofisso} holds, see~\cite[Theorem~1.1(1)]{Kieran:numhilb2}. 

We show below that if the complete linear system $|L|$ behaves well in a neighborhood of $F$, then the equality in~\eqref{canonicofisso} holds. We start with an intermediate result.
Since the fixed locus of $\tau_{\lambda}$ is irreducible, there is a unique $\mu_2$-linearization of $L$ which lifts the action of $\la\tau_{\lambda}\ra$ and acts as multiplication by $+1$ on $F$. Then the quotient $L/\langle \tau_\lambda\rangle$ is a line bundle on $X/\la\tau_{\lambda}\ra$ and  pulls back to $L$ via the quotient map $X\to X/\la\tau_{\lambda}\ra$.

\begin{prop}\label{prop:sezinv}
Let $(X,\lambda)$ be a polarized $\mathrm{HK}$ manifold of $\mathrm{K3}^{[n]}$-type with $q_X(\lambda)=2$ and $\divisore(\lambda)=1$. 
Let $L$ be the ample line bundle such that $c_1(L)=\lambda$, equipped with the above $\mu_2$-linearization.  
Then $H^0(X,L)^{\mu_2}=H^0(X,L)$. 
If $(X,\lambda)$ is general, then $L$ is globally generated.
\end{prop}

\begin{proof}
Let $(S,h)$ be a polarized K3 surface of genus $2$ with $\Pic(S)=\ZZ h$. 
Let $f\colon S\to\PP^2$  be the associated double cover ramified over a smooth sextic curve $\Gamma\subset \PP^2$ and let $\tau_S$ be the covering involution on $S$. 
Then $(S^{(n)},h^{(n)})$ is a specialization of $(X,\lambda)$, where $h^{(n)}$ is obtained by symmetrizing $h$. 
The involution $\tau_S$ induces an involution $\tau_0$ of $S^{(n)}$ which has one dimensional $+1$ eigenspace in $\NS(S^{(n)})$, generated by $h^{(n)}$.  
Since $\tau_{\lambda}$ has one dimensional $+1$ eigenspace in $\NS(X)$ generated by $\lambda$, it follows that the involution $\tau_0$ on $S^{(n)}$ is the specialization of $\tau_{\lambda}$.   
Moreover $S^{(n)}/\la\tau_0\ra$ is a specialization of $X/\la\tau_{\lambda}\ra$, and the specialization of $L/\langle \tau_\lambda\rangle$ is the line bundle $\ov{L}_0$ obtained by pulling back $(\cO_{\PP^2}(1))^{(n)}$
via the natural map
\begin{equation*}
S^{(n)}/\la\tau_0\ra\lra(\PP^2)^{(n)}.
\end{equation*}
We have
\begin{equation}\label{stringa}
h^0(X,L)=\binom{n+2}{2}=h^0\left((\PP^2)^{(n)},(\cO_{\PP^2}(1))^{(n)}\right)=h^0\left(S^{(n)},\cO_{S^{(n)}}(h^{(n)})\right).
\end{equation}
Indeed, the first equality follows from Kodaira vanishing and the well-known explicit HRR formula for line bundles on {\rm HK} manifolds of $\mathrm{K3}^{[n]}$-type, see for example~\cite[Eqn.~(5)]{Debarre:Survey}.
Hence every section of $\cO_{S^{(n)}}(h^{(n)})$ extends to a section of  $L$ on $X$.
Since $\cO_{S^{(n)}}(h^{(n)})$ is globally generated, it follows that  $L$ is globally generated for general $(X,\lambda)$. 

It remains to prove that $H^0(X,L)^{\mu_2}=H^0(X,L)$. By the last equality in~\eqref{stringa} we have 
\begin{equation}\label{strega}
h^0(X,L)=h^0\left(S^{(n)}/\la\tau_0\ra,\ov{L}_0\right).
\end{equation}
Since  $\ov{L}_0$ is ample, all its higher cohomology vanishes by Kawamata-Viehweg vanishing.
Since the quotient $S^{(n)}/\la \tau_0\ra$ is a specialization of $X/\la\tau_{\lambda}\ra$, it follows that every section  of $\ov{\lambda}_0$ extends to a section of $\ov{L}_0$. This proves that  $H^0(X,L)^{\mu_2}=H^0(X,L)$ by the equality in~\eqref{strega}.
\end{proof}

\begin{rema}
Proposition~\ref{prop:ChoiceLinearizationFamily} and Proposition~\ref{prop:sezinv}
suggest that $\Fix(\tau_{\lambda})$ in the case $\divisore(\lambda)=1$ should be analogous to $\Fix(\tau_{\lambda})_{+}$ in the case $\divisore(\lambda)=2$. 
\end{rema}

By Proposition~\ref{prop:sezinv}
we have a regular factorization of $\psi\colon X\to |L|^{\vee}$ as
\begin{equation}\label{fattorizzo}
X\to X/\longrightarrow \tau_{\lambda}\ra \overset{\ov{\psi}}{\lra}  |L|^{\vee}\cong\PP^{\frac{n(n+3)}{2}}.
\end{equation}

\begin{prop}
Keep notation as above, and assume that there exists an open subset 
$F_0\subset F\subset X/\la \tau_{\lambda}\ra$ (abusing notation we denote by $F$ the image of $F$ in $X/\la \tau_{\lambda}\ra$) with complement of codimension at least $2$ such that in a neighborhood of $F_0$ the map $\ov{\psi}$ is an isomorphism onto its image. Then the equality in~\eqref{canonicofisso} holds. 
\end{prop}

\begin{proof}
Set $N:=n(n+3)/2$. To simplify notation we denote by the same symbol $F_0$ and its image in $|L|^{\vee}\cong\PP^{N}$. (This makes sense because the image of $F_0$ is isomorphic to $F_0$ by hypothesis.)  The key observation is that the following isomorphism holds:
\begin{equation}\label{chiave}
\cN^{\vee}_{F_0/\PP^N}\cong \Sym^2(\cN^{\vee}_{F_0/X}).
\end{equation}
In fact, because of the factorization in~\eqref{fattorizzo}, the pullback by $\psi$ gives an injection $\psi^{*}\cI_{F_0/\PP^N}\subset \cI_{F_0/X}^2$. Restricting to $F_0$, and using our hypothesis that $\ov{\psi}$ is an isomorphism onto its image  in a neighborhood of $F_0$, we get a surjection
\begin{equation*}
\cN_{F_0/\PP^N}^{\vee}\twoheadrightarrow \Sym^2(\cN^{\vee}_{F_0/X})
\end{equation*}
A straightforward computation shows that the two vector bundles have the same rank, hence they are isomorphic. This proves that we have the isomorphism in~\eqref{chiave}. 

Since $F$ is Lagrangian, a symplectic form on $X$ defines an isomorphism $\cN^{\vee}_{F_0/X}\cong T_X$. By the normal exact sequence of $F_0$ in $\PP^N$ and the isomorphism in~\eqref{chiave} we get the following equalities in $\CH^1(F_0)$:
\begin{equation*}
-(N+1) c_1(L|_{F_0})=c_1(\omega_{F_0})+c_1(\cN^{\vee}_{F_0/\PP^N})=
c_1(\omega_{F_0})+c_1(\Sym^2 T_X).
\end{equation*}
A straightforward computation gives that the equality in~\eqref{canonicofisso} holds. 
\end{proof}



\begin{thebibliography}{LLSvSS17}

\bibitem[Alp13]{Alper:GoodModuli} Alper, J., Good moduli spaces for Artin stacks, {\it  Ann. Inst. Fourier (Grenoble)} {\bf 63} (2013), 2349--2402.

\bibitem[AH-LH23]{AHLH:Moduli} Alper, J., Halpern-Leistner, D., Heinloth, J., Existence of moduli spaces for algebraic stacks, {\it Invent. Math.} {\bf 234} (2023), 949–-1038.

\bibitem[AS18]{AS:singularities} Arbarello, E., Sacc\`a, G., Singularities of moduli spaces of sheaves on K3 surfaces and Nakajima quiver varieties, {\it Adv. Math.} {\bf 329} (2018), 649--703.

\bibitem[AS25]{AS:Formality} \bysame, Singularities for Bridgeland moduli spaces for K3 categories: an update, in {\it Perspectives on four decades of algebraic geometry. Vol. 1. In memory of Alberto Collino}, 1--42, Progr.~Math. {\bf 351}, Birkh\"auser/Springer, Cham, 2025.

\bibitem[BD85]{BD:Fano} La vari\'et\'e des droites d'une hypersurface cubique de dimension 4, \textit{CR Acad. Sci. Paris Sér. I Math 301}{\bf 14 } (1985),703--706.

\bibitem[BMM21]{BMM:Formality} Bandiera, R., Manetti, M., Meazzini, F., Formality conjecture for minimal surfaces of Kodaira dimension 0, {\it Compos. Math.} {\bf 157} (2021), 215--235.

\bibitem[BMM22]{BMM:Quadraticity} \bysame, Deformations of polystable sheaves on surfaces: quadraticity implies formality, {\it Mosc. Math. J.} {\bf 22} (2022), 239--263.

\bibitem[BL+21]{BLMNPS:family} Bayer, A., Lahoz, M., Macr\`i, E., Nuer, H., Perry, A., Stellari, P., Stability conditions in families, {\it Publ. Math. IH\'ES} {\bf 133} (2021), 157--325.

\bibitem[BM14]{BM:walls} Bayer, A., Macr\`i, E., MMP for moduli of sheaves on K3s via wall-crossing: nef and movable cones, Lagrangian fibrations, {\it Invent.~Math.}  {\bf 198} (2014), 505--590.

\bibitem[Bec25]{beck:Atomic} Beckmann, T., Atomic objects on hyper-K\"ahler manifolds, {\it J.~Algebraic Geom.} {\bf 34} (2025), 109--160.

\bibitem[Bri07]{Bridgeland:Stab} Bridgeland, T., Stability conditions on triangulated categories. {\it Ann. of Math. (2)} {\bf 166} (2007), 317--345.

\bibitem[BZ19]{BZ:Kuranishi} Budur, N., Zhang, Z., Formality conjecture for K3 surfaces, {\it Compos. Math.} {\bf 155} (2019), 902--911.

\bibitem[CCL22]{CCL:Chow} Camere, C., Cattaneo, A., Laterveer, R.,
On the Chow ring of certain Lehn--Lehn--Sorger--van Straten eightfolds, {\it Glasg. Math. J.} {\bf 64} (2022), 253--276.

\bibitem[CS17]{CanonacoStellari:Tour} Canonaco, A., Stellari, P., A tour about existence and uniqueness of dg enhancements and lifts, {\it J. Geom. Phys.} {\bf 122} (2017), 28--52.

\bibitem[CS18]{CanonacoStellari:Uniqueness} \bysame, Uniqueness of dg enhancements for the derived category of a Grothendieck category, {\it J. Eur. Math. Soc. (JEMS)} {\bf 20} (2018), 2607--2641.

\bibitem[Cas18]{Castravet:MDS} Castravet, A.-M., Mori dream spaces and blow-ups, in \emph{Algebraic geometry: Salt Lake City 2015}, 143--167, Proc. Sympos. Pure Math. {\bf 97.1}, Amer. Math. Soc., Providence, RI, 2018. 

\bibitem[CPZ24]{PZ:FormalityBridgeland} Chen, H., Pertusi, L., Zhao, X., Some remarks about deformation theory and formality conjecture, eprint {\tt arXiv:2402.06579}.

\bibitem[C-B01]{CB:geometry} Crawley-Boevey, W., Geometry of the moment map for representations of quivers, {\it Compositio Math.} {\bf 126} (2001), 257--293.

\bibitem[C-B03]{CB:normality} \bysame, Normality of Marsden--Weinstein reductions for representations of quivers, {\it Math. Ann.} {\bf 325} (2003), 55--79.

\bibitem[Deb22]{Debarre:Survey} Debarre, O., Hyper-K\"ahler manifolds,  {\it Milan J. Math.} {\bf 90} (2022), 305--387.

\bibitem[DK18]{DK: GM varieties} Debarre, O., Kuznetsov, A., Gushel--Mukai varieties: classification and birationalities, {\it Algebr. Geom.} {\bf 5} no. 1 (2018), 15--76.

\bibitem[DV10]{DV: varieties} Debarre, O., Voisin, C., Hyper-K\"ahler fourfolds and Grassmann geometry, {\it J. Reine Angew. Math.}, {\bf 649} (2010), 63--87.

\bibitem[DLP85]{DLP} Dr\'ezet, J.-M., Le Potier, J., Fibr\'es stables et fibr\'es exceptionnels sur $\mathbf{P}^2$, {\it Ann. Sci. \'Ecole Norm. Sup. (4)} {\bf 18} (1985), 193--243.

\bibitem[Fal80]{Faltings:Formal} Faltings, G., Some theorems about formal functions, {\it Publ.~Res.~Inst.~Math.~Sci.} {\bf 16} (1980), 721--737.

\bibitem[FM21]{FM:Fano} Fatighenti, E., Mongardi, G., Fano varieties of K3-type and IHS manifolds,  {\it Int. Math. Res. Not. IMRN} {\bf 4} (2021), 3097--3142.

\bibitem[Fer11]{Ferretti:thesis} Ferretti, A.,
The Chow ring of double EPW sextics, {\it Rend. Mat. Appl. (7)} {\bf 31} (2011), 69--217.

\bibitem[FIM12]{FIM:dgla} Fiorenza, D., Iacono, D., Martinengo, E., Differential graded Lie algebras controlling infinitesimal deformations of coherent sheaves, {\it J. Eur. Math. Soc. (JEMS)} {\bf 14} (2012), 521--540.

\bibitem[FM+22]{involutions1} Flapan, L., Macr\`i, E., O'Grady, K., Sacc\`a, G., The geometry of antisymplectic involutions, I,  {\it Math. Z.} {\bf 300} (2022), 3457--3495.


\bibitem[HT10]{HT:IntersectionNumbers} Hassett, B., Tschinkel, Y., Intersection numbers of extremal rays on holomorphic symplectic varieties, {\it Asian J. Math.} {\bf 14} (2010), 303--322.

\bibitem[HL10]{HL:moduli} Huybrechts, D., Lehn, M., {\it The geometry of moduli spaces of sheaves}, Cambridge Mathematical Library, Cambridge University Press, Cambridge, 2010.

\bibitem[IM15]{IM:Fano} Iliev, A., Manivel, L.,  Fano manifolds of Calabi--Yau Hodge type, \textit{Journal of Pure and Applied Algebra}, {\bf 219} no. 6 (2015) 2225--2244.

\bibitem[IM19]{IM: magic square} \bysame, Hyperk\"ahler manifolds from the Tits-–Freudenthal magic square, {\it European Journal of Mathematics} {\bf 5} (2019), 1139--1155.

\bibitem[Kin94]{King:Moduli} King, A., Moduli of representations of finite-dimensional algebras, {\it Quart. J. Math. Oxford Ser. (2)} {\bf 45} (1994), 515--530. 

\bibitem[Kle05]{Kleiman:Picard} Kleiman, S., The Picard scheme, in {\it Fundamental algebraic geometry}, 235--321, Math. Surveys Monogr. {\bf 123}, Amer. Math. Soc., Providence, RI, 2005.

\bibitem[Kol13]{Kollar:SingularitiesMMP} Koll\'ar, J., {\it Singularities of the minimal model program}, Cambridge Tracts in Mathematics {\bf 200}, Cambridge University Press, Cambridge, 2013.

\bibitem[Kuh61]{Kuh} Kuhlmann, N., Die Normalisierung komplexer R\"aume, {\it Math. Ann.} {\bf 144} (1961), 110--125.

\bibitem[LeP92]{LePotier:DetLineBundle} Le Potier, J., Fibr\'e d\'eterminant et courbes de saut sur les surfaces alg\'ebriques, in {\it Complex projective geometry (Trieste, 1989/Bergen, 1989)}, 213--240, London Math. Soc. Lecture Note Ser. {\bf 179}, Cambridge University Press, Cambridge, 1992.


\bibitem[Leh15]{Lehn:Oberwolfach} Lehn, M., Twisted cubics on a cubic fourfold and in involution on the associated 8-dimensional symplectic manifold, in {\it Oberwolfach Report} {\bf 51/2015}, 22--24, 2015.

\bibitem[LLSvS17]{LLSvS:cubics} Lehn, C., Lehn, M., Sorger, C., van Straten, D., Twisted cubics on cubic fourfolds, {\it J. Reine Angew. Math.} {\bf 731} (2017), 87--128.

\bibitem[LPZ23]{LPZ: twisted} Li, C., Pertusi, L., Zhao, X., Twisted cubics on cubic fourfolds and stability conditions, {\it Algebr. Geom.} {\bf 10} (2023), 620--642.

\bibitem[Li93]{Li:uhlencomp} Li, J., Algebraic geometric interpretation of Donaldson's polynomial invariants, {\it J. Differential Geom.} {\bf 37} (1993), 417--466.

\bibitem[Lie06]{Lieblich:moduli} Lieblich, M., Moduli of complexes on a proper morphism, {\it J.\ Algebraic Geom.} {\bf 15} (2006), 175--206.

\bibitem[LO10]{LO:Uniqueness} Lunts, V., Orlov, D., Uniqueness of enhancement for triangulated categories, {\it J. Amer. Math. Soc.} {\bf 23} (2010), 853--908.

\bibitem[MS17]{MS:lectures} Macr\`i, E., Schmidt, B., Lectures on Bridgeland stability, in {\it Moduli of curves}, 139--211, Lect. Notes Unione Mat. Ital. {\bf 21}, Springer, Cham, 2017.

\bibitem[Man15]{Manetti:FormalityCriteria} Manetti, M., On some formality criteria for DG-Lie algebras, {\it J. Algebra} {\bf 438} (2015), 90--118.

\bibitem[Mar10]{markman:modular} Markman, E., Modular Galois covers associated to symplectic resolutions of singularities, {\it J.~Reine Angew.~Math.} {\bf 644} (2010), 189--220.


\bibitem[Muk97]{Mukai:Tata} Mukai, S., On the moduli space of bundles on K3 surfaces. I, in {\it Vector bundles on algebraic varieties (Bombay, 1984)}, 341--413, Tata Inst. Fund. Res. Stud. Math. {\bf 11}, Tata Inst. Fund. Res., Bombay, 1987.

\bibitem[Nak99]{Nakajima:book} Nakajima, H., {\it Lectures on Hilbert schemes of points on surfaces}, Univ. Lecture Ser. {\bf 18}, American Mathematical Society, Providence, RI, 1999.

\bibitem[Nam01]{Namikawa:Def} Namikawa, Y., Deformation theory of singular symplectic n-folds, {\it Math.~Ann.} {\bf 319} (2001), 597--623.

\bibitem[O'G97]{Kieran:weight2} O'Grady, K., The weight-two Hodge structure of moduli spaces of sheaves on a $K3$ surface, {\it J. Algebraic Geom.} {\bf 6} (1997), 599--644.

\bibitem[O'G08]{Kieran:numhilb2} \bysame, Irreducible symplectic 4-folds numerically equivalent to $K3^{[2]}$, {\it Commun. Contemp. Math.} {\bf 10} (2008), 553--608. 

\bibitem[O'G17]{Kieran:LagrangianCovering} \bysame, Covering families of Lagrangian subvarieties, preprint 2017.

\bibitem[Ohk10]{Okhawa:P2} Ohkawa, R., Moduli of Bridgeland semistable objects on $\mathbb{P}^2$, {\it Kodai Math. J.} {\bf 33} (2010), 329--366.

\bibitem[PPZ22]{PPZ: GM} Perry, A., Pertusi, L., Zhao, X.,
Stability conditions and moduli spaces for Kuznetsov components of Gushel--Mukai varieties, {\it Geom. Topol.} {\bf 26} (2022), 3055-–3121.

\bibitem[PVdB19]{PolVdB:equivariant} Polishchuk, A., Van den Bergh, M.,
Semiorthogonal decompositions of the categories of equivariant coherent sheaves for some reflection groups, {\it J. Eur. Math. Soc. (JEMS)} {\bf 21} (2019), 2653--2749.

\bibitem[Pro76]{procesi:invariants} Procesi, C., The invariant theory of $n\times n$ matrices, {\it Adv. Math.} {\bf 19} (1976), 306--381.

\bibitem[StP25]{stacks-project} The Stacks Project Authors, The Stacks Project, 2025, available at \url{http://stacks.math.columbia.edu}.

\bibitem[Taj23]{Tajakka:Uhlenbeck} Tajakka, T., Uhlenbeck Compactification as a Bridgeland Moduli Space, {\it Int. Math. Res. Not. IMRN} (2023), 4952--4997.

\bibitem[Yos01]{Yos:moduli} Yoshioka, K., Moduli spaces of stable sheaves on abelian surfaces, {\it Math. Ann.} {\bf 321} (2001), 817--884.

\end{thebibliography}
\end{document}